\documentclass{amsart}
\usepackage{ amssymb }
\usepackage{bm}
\usepackage[margin=1in]{geometry}
\usepackage{hyperref}       
\usepackage{url}            
\usepackage{booktabs}       
\usepackage{amsfonts}       
\usepackage{nicefrac}       
\usepackage{microtype}      
\usepackage{blkarray}
\usepackage{lipsum}
\usepackage{graphicx}
\usepackage{nccmath}
\usepackage{caption}
\usepackage{listings}
\usepackage{xcolor}

\def\qed{\hfill {\hbox{${\vcenter{\vbox{               
   \hrule height 0.4pt\hbox{\vrule width 0.4pt height 6pt
   \kern5pt\vrule width 0.4pt}\hrule height 0.4pt}}}$}}}

\newtheorem{theorem}{Theorem}
\newtheorem{definition}{Definition}
\newtheorem{lemma}[theorem]{Lemma}
\newtheorem{proposition}[theorem]{Proposition}

\DeclareMathOperator{\ord}{ord}

\newenvironment{mpmatrix}{\begin{medsize}\begin{pmatrix}}%
{\end{pmatrix}\end{medsize}}%

\title{Petal Projections, Knot Colorings \& Determinants}
 
\author{Allison Henrich}
\address{Department of Mathematics, Seattle University, 901 12th Avenue, Seattle, WA 98122}
\email{henricha@seattleu.edu}

\author{Robin Truax}
\address{Stanford University, Stanford, CA 94305}
\email{truax@stanford.edu}



\begin{document}
\maketitle

\begin{abstract}
An \"{u}bercrossing diagram is a knot diagram with only one crossing that may involve more than two strands of the knot. Such a diagram without any nested loops is called a petal projection. Every knot has a petal projection from which the knot can be recovered using a permutation that represents strand heights. Using this permutation, we give an algorithm that determines the $p$-colorability and the determinants of knots from their petal projections. In particular, we compute the determinants of all prime knots with crossing number less than $10$ from their petal permutations. 
\end{abstract}

\section{Background}
There are many different ways to define a knot. The standard definition of a {\em knot}, which can be found in \cite{adams0} or \cite{allison-book-:D}, is an embedding of a closed curve in three-dimensional space, $K: S^1 \rightarrow \mathbb{R}^3$. Less formally, we can think of a knot as a knotted circle in 3-space. Though knots exist in three dimensions, we often picture them via 2-dimensional representations called \textit{knot diagrams}. In a knot diagram, crossings involve two strands of the knot, an overstrand and an understrand. The relative height of the two strands at a crossing is conveyed by putting a small break in one strand (to indicate the understrand) while the other strand continues over the crossing, unbroken (this is the overstrand). Figure \ref{fig:figure-eight} shows an example of this.\\

\begin{figure}
    \centering
    \includegraphics[width=0.25\textwidth]{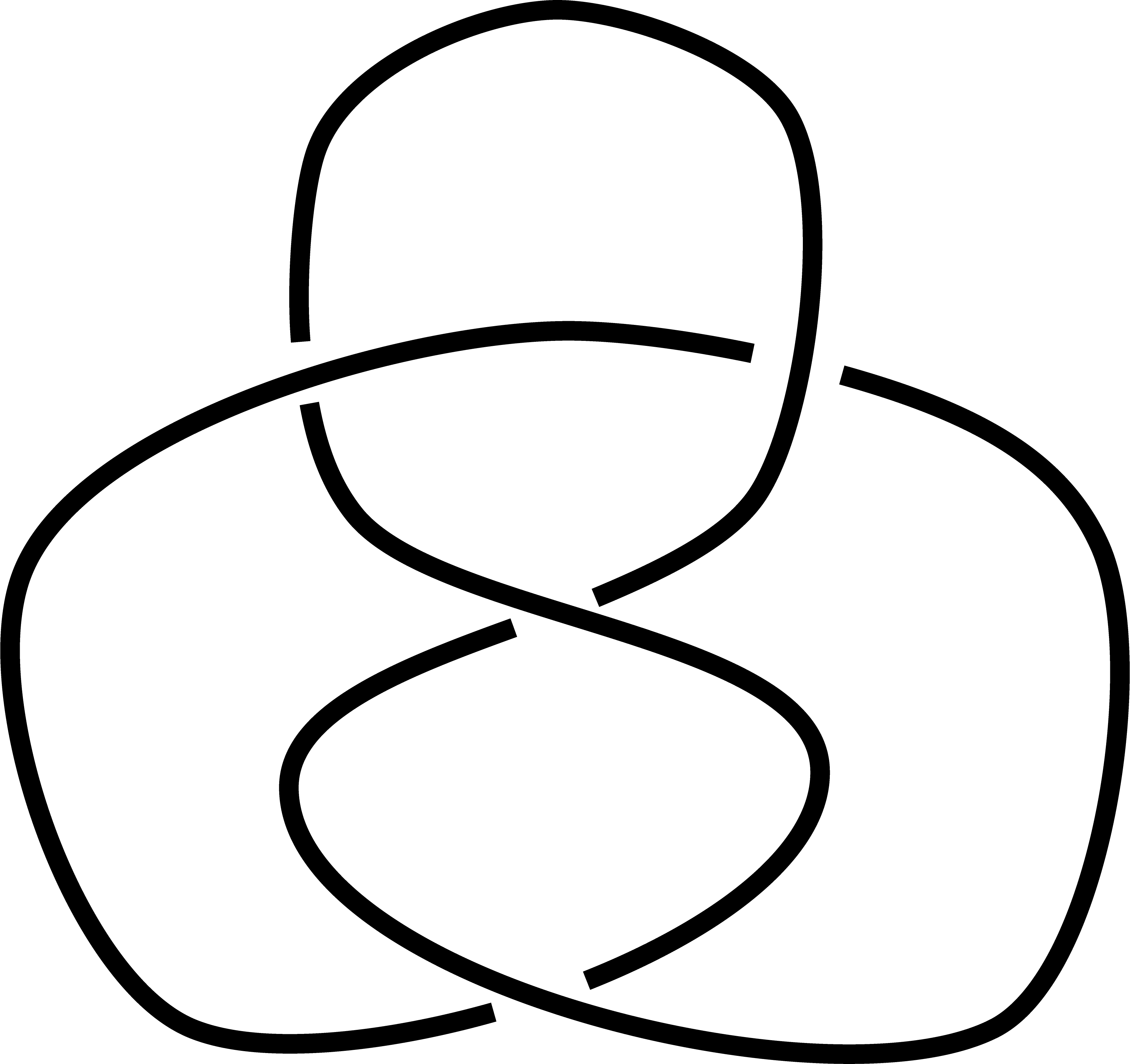}
    \caption{A knot diagram.}
    \label{fig:figure-eight}
\end{figure}

Two knots are defined to be equivalent if one can be continuously deformed (without passing through itself) into the other. In terms of diagrams, two knot diagrams represent equivalent knots if and only if they can be related by a sequence of Reidemeister moves and planar isotopies (i.e. ``wiggling"). In Figure \ref{fig:r-moves}, we illustrate the three Reidemeister moves. To learn more, see \cite{adams0} or \cite{allison-book-:D}. \\ 

\begin{figure}
    \centering
    \includegraphics[width=0.50\textwidth]{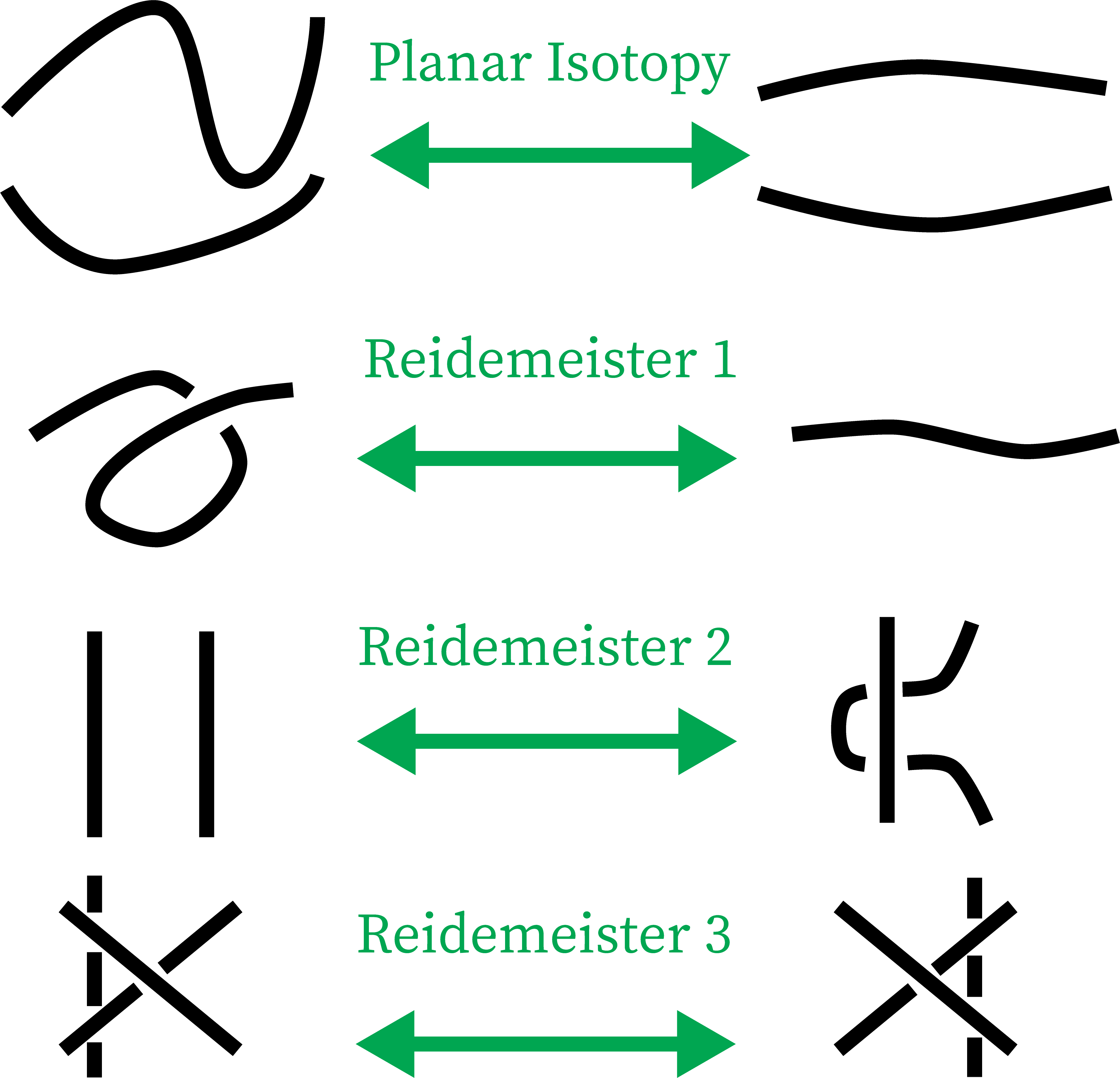}
    \caption{The three Reidemeister moves.}
    \label{fig:r-moves}
\end{figure}

If a knot is equivalent to a simple, crossingless circle, it is called a \textit{trivial knot}, or  the \textit{unknot}. Two examples of the unknot are provided in Figure \ref{fig:unknot}.\\

\begin{figure}
    \centering
    \includegraphics[width=0.25\textwidth]{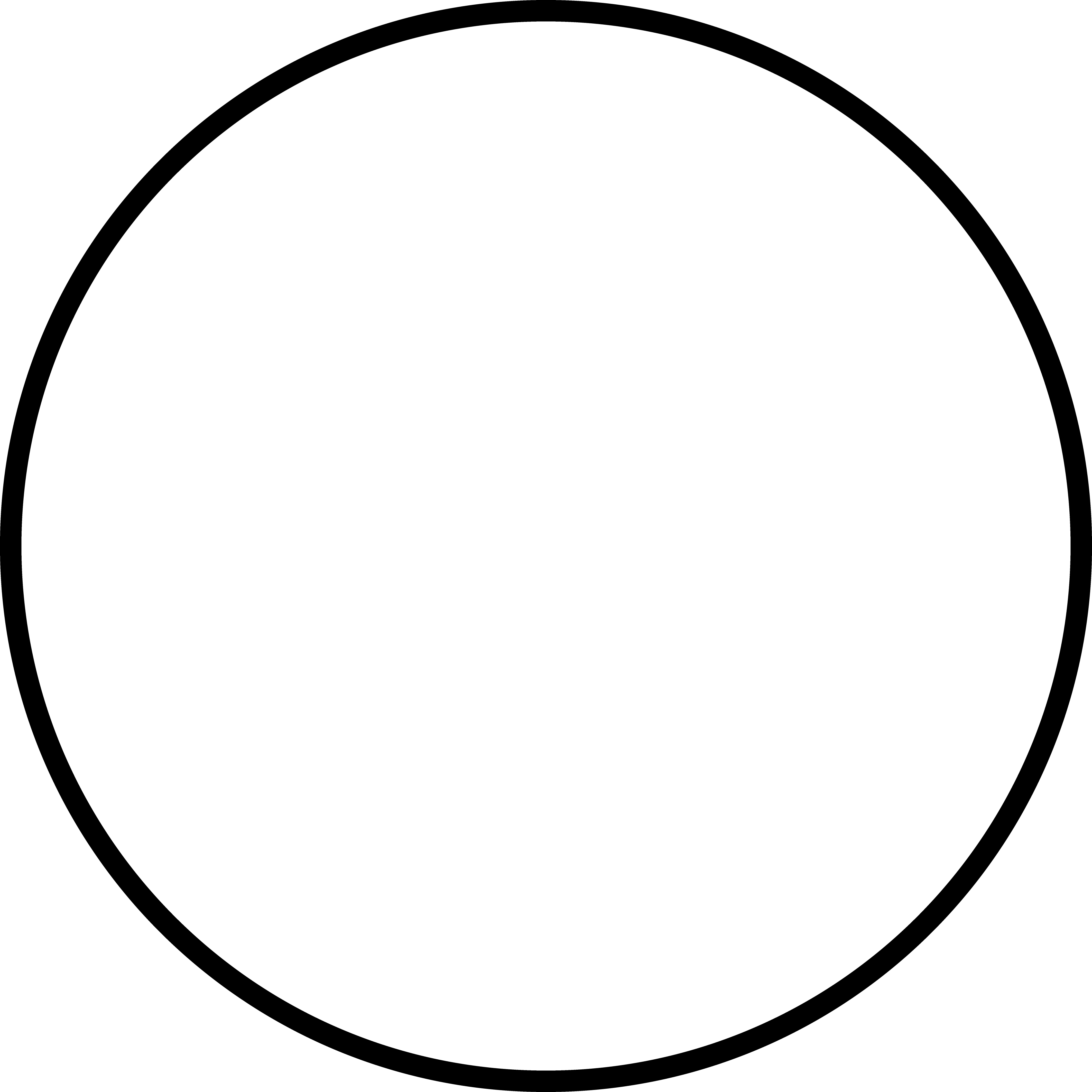}
    \hspace{.7in}
    \includegraphics[width=0.30\textwidth]{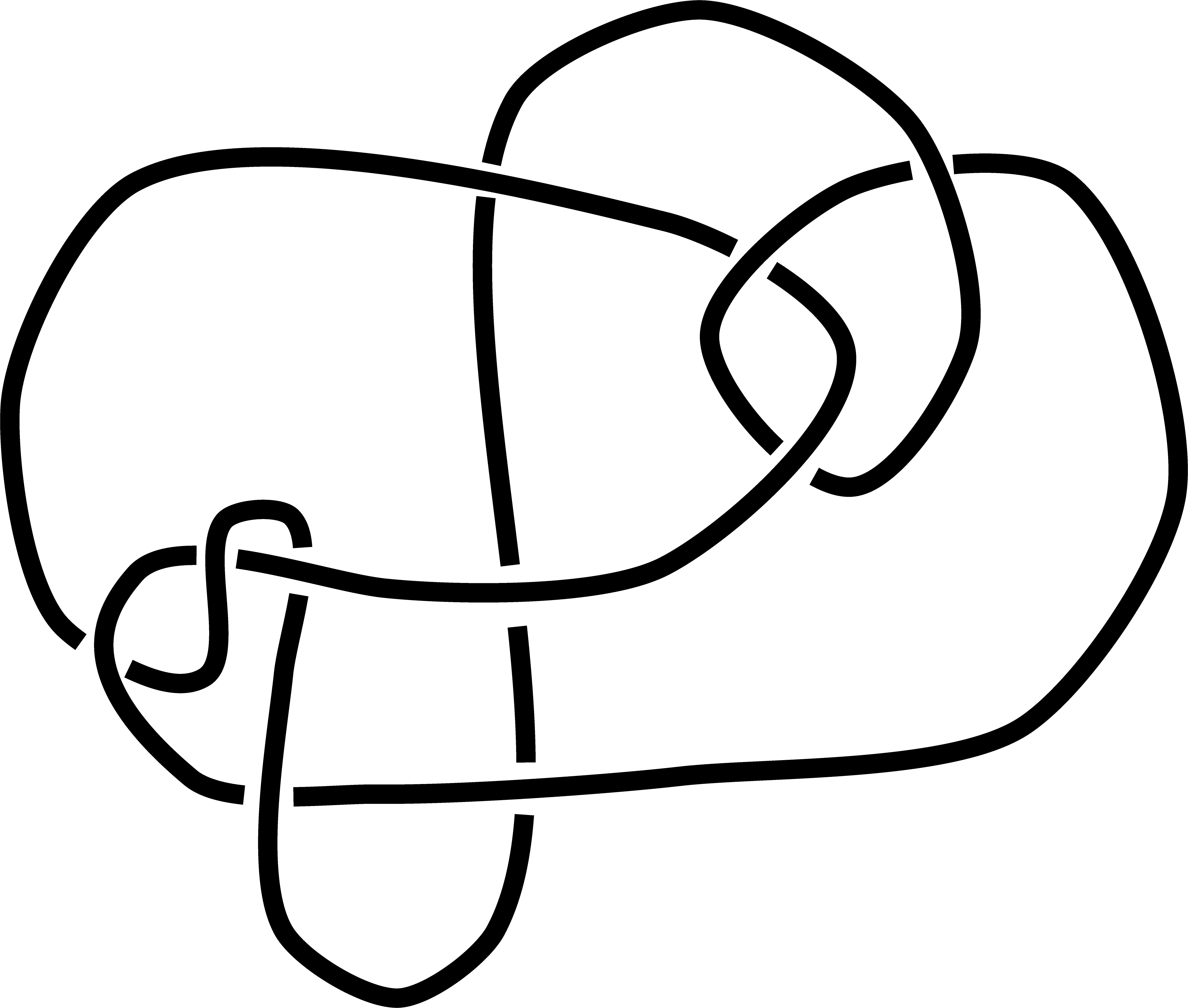}
    \caption{Two diagrams of the unknot.}\label{fig:unknot}
\end{figure}

If a knot cannot be deformed into the unknot, it is said to be \textit{non-trivial}---that is, there's no way to untangle it without passing it through itself. One of the central problems of knot theory is to find ways to show that a knot is non-trivial. To do this, much of knot theory is based on finding and understanding \textit{knot invariants}---functions defined on knots that give the equivalent outputs (in the form of numbers, polynomials, etc.) if the two input knots are equivalent. Examples of knot invariants include colorability, the knot determinant, the Jones polynomial, the Alexander polynomial, and Khovanov homology. It is the first two of these invariants that will be the focus of this paper. 

\section{Colorability}
This section relies heavily on an important definition in knot theory: an ``arc." 
\begin{definition}\label{arc}
An {\bf arc} is a segment of a knot in a diagram which begins and ends at adjacent undercrossings (crossings such that the segment passes under the crossing). 
\end{definition}

\begin{figure}
    \centering  
    \includegraphics[width=0.30\textwidth]{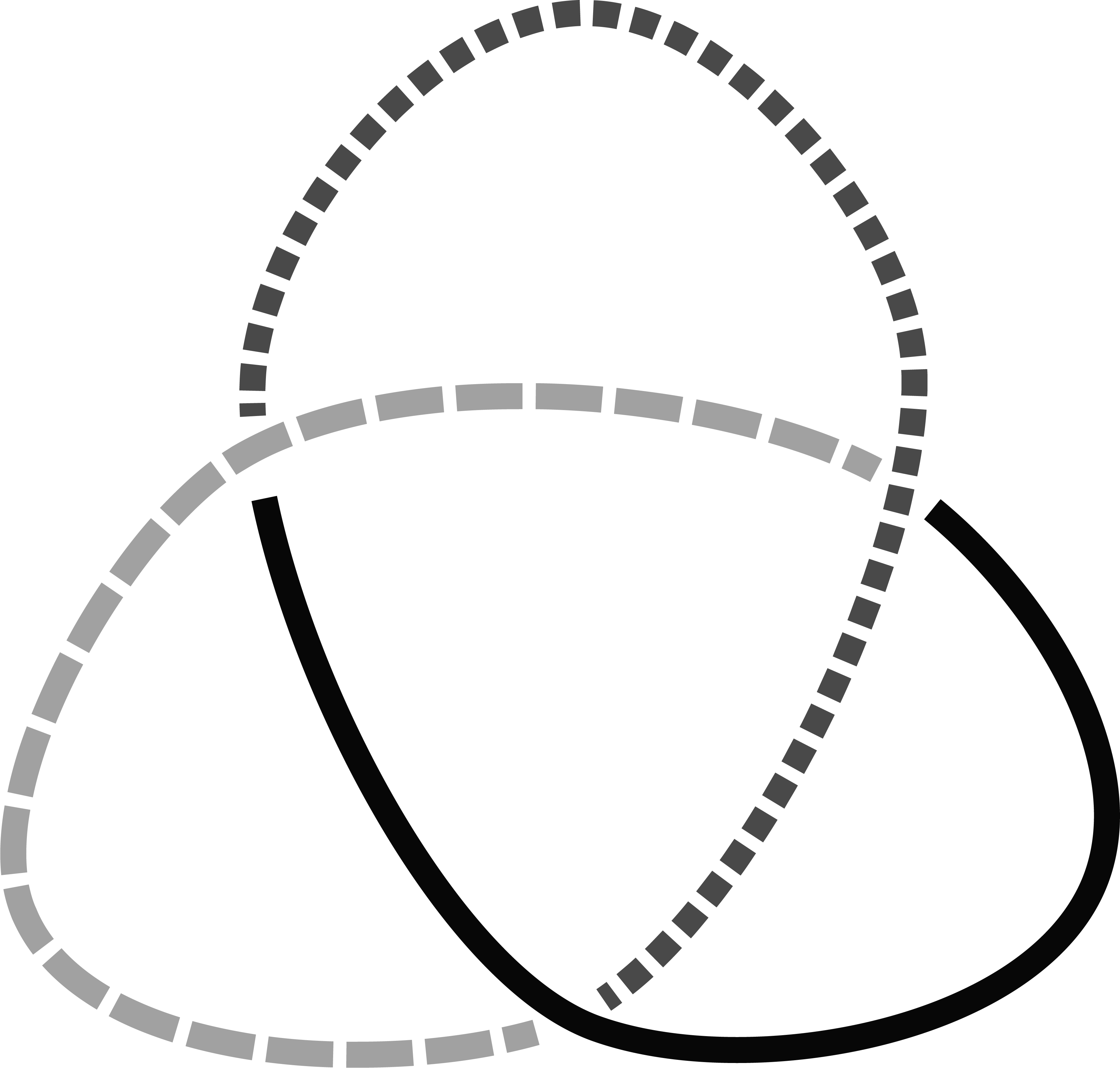}
    \caption{A decomposition of the trefoil into $3$ arcs (dashes used to distinguish different arcs, not actual breaks in the knot).}
    \label{fig:trefoil-arcs}
\end{figure}

We may decompose a knot diagram into a collection of arcs. Figure \ref{fig:trefoil-arcs} demonstrates an example. It's also important to note that there are always the same number of arcs as there are crossings, with the exception of the crossingless diagram of the unknot. We can see this because at every crossing, exactly one of the strands must go under another, thereby starting a new arc. This will become important later, as it lets us construct a matrix from a knot diagram (with rows corresponding to crossings and columns corresponding to arcs) that is guaranteed to be square.

\subsection{Tricolorability}
To understand the general notion of $p$-colorability---one of the knot invariants we are most interested in for the purposes of this paper---we should first describe the specific case when $p=3$: \textit{tricolorability}. Colorability is a well-studied concept in knot theory (see \cite{livingston} for a further description of all the topics covered in this section), and tricolorability is a particularly accessible type of colorability to study.
\begin{definition}\label{tricolorable}
A knot diagram is {\bf tricolorable} when each of the arcs in the knot diagram can be assigned one of three colors such that: 
\begin{itemize}
    \item more than one color is used, and
    \item at any given crossing, all of the arcs are either the same color or all different colors. 
\end{itemize}
\end{definition}

\begin{figure}
    \centering
    \includegraphics[width=0.25\textwidth]{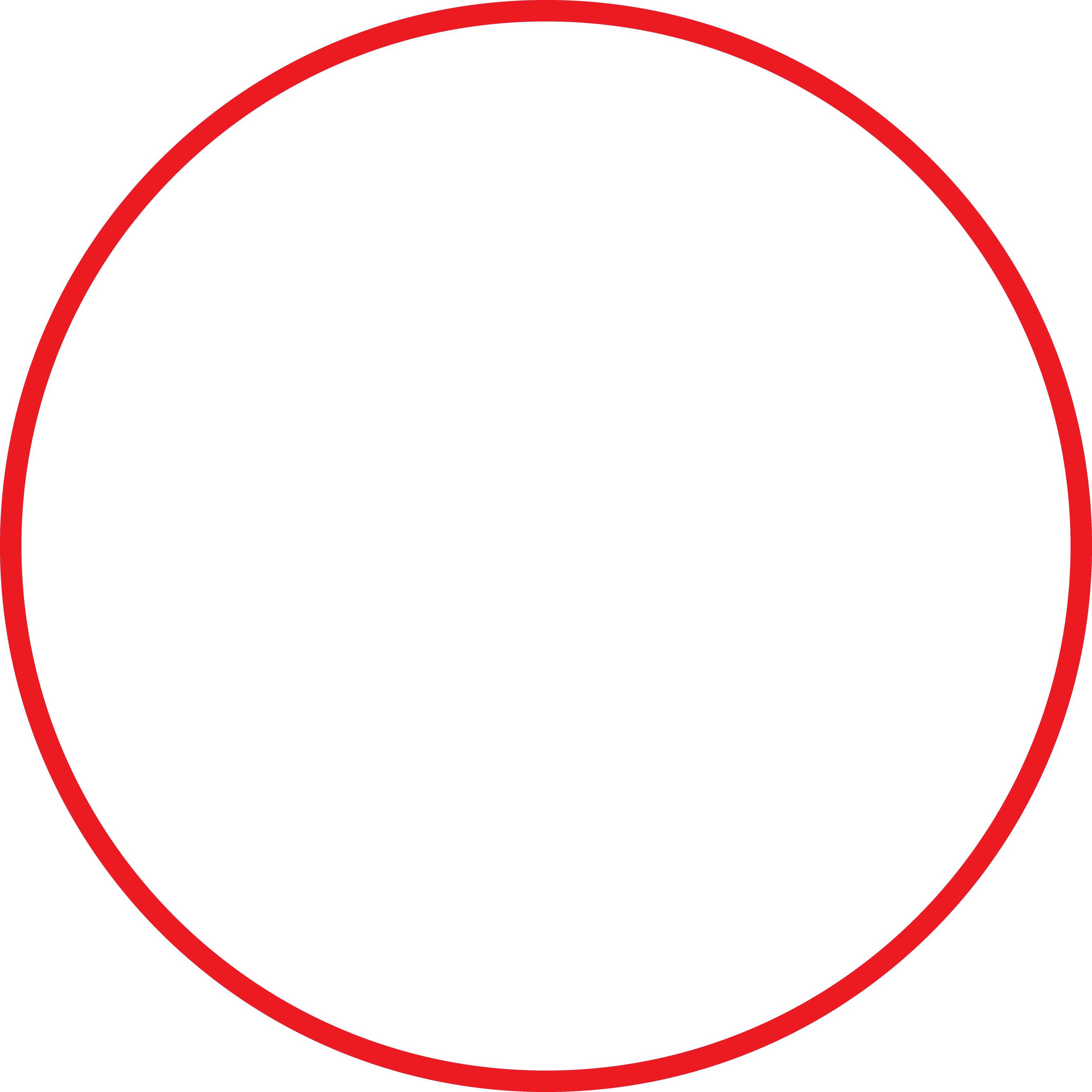}
    \hspace{.7in}
    \includegraphics[width=0.25\textwidth]{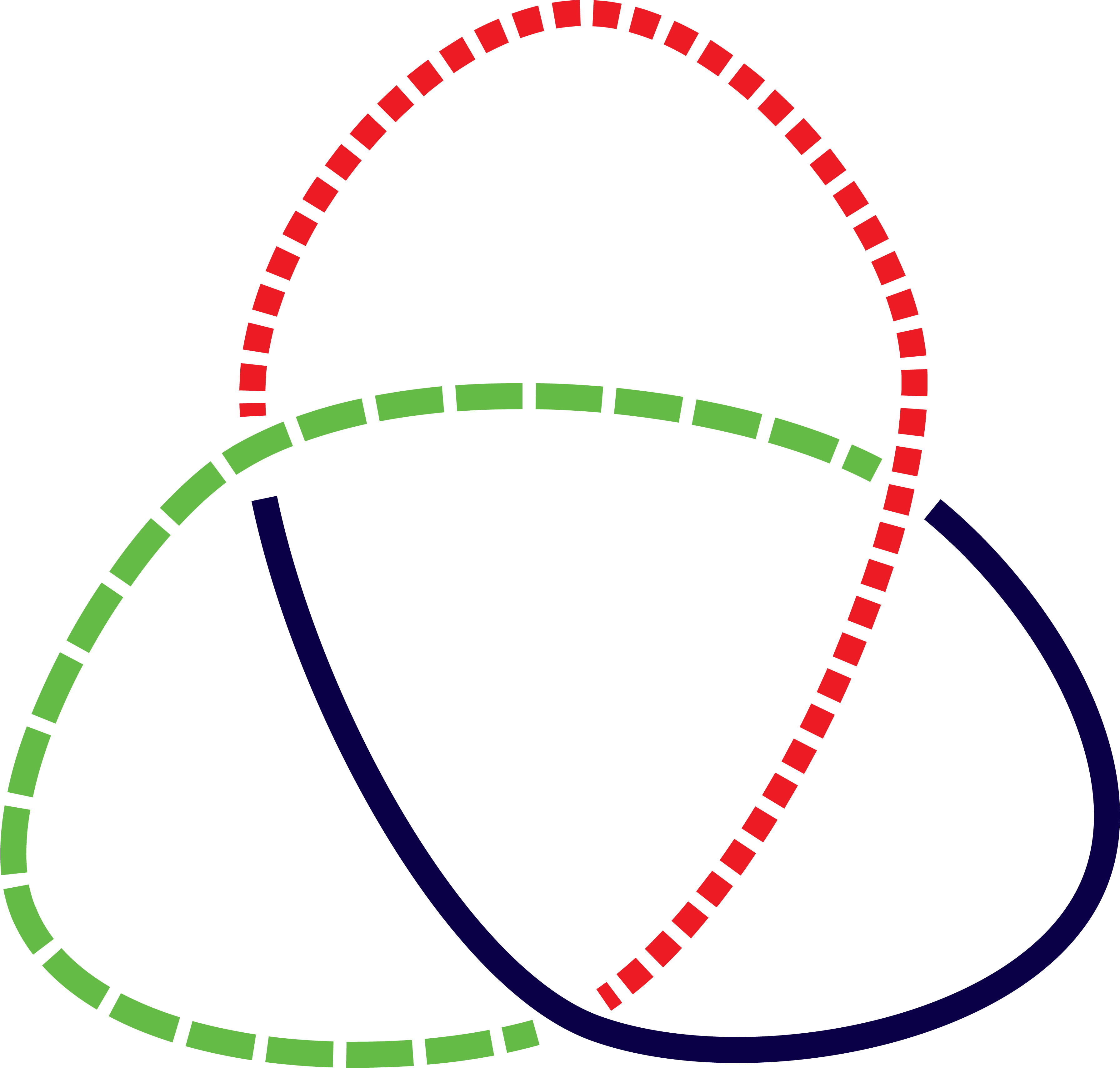}
    \caption{The unknot (left) is not tricolorable, while the trefoil (right) is.}
    \label{fig:coloring}
\end{figure}

For example, the usual diagram of the unknot is not tricolorable, since it has only one arc and therefore can only use one color. In contrast, the usual diagram of trefoil knot is tricolorable; see Figure \ref{fig:coloring}. Not only is it true that tricolorability is a knot invariant -- that is, any two diagrams of a knot are either both tricolorable or both not tricolorable -- but the number of ways to tricolor a knot is also invariant under all Reidemeister moves. Therefore, any two equivalent knot diagrams must be tricolorable in the same number of ways. This means that showing that there are a different number of ways to tricolor two knots is enough to show that they're not equivalent. Note that the converse is \textit{not} true---two knots that are tricolorable in the same number of ways may still be distinct, so the invariant isn't complete. \\

Before we talk about generalizing tricolorability, let us see how determining if a knot is tricolorable can be done by solving a system of linear equations. Suppose we assign a variable to each arc in a knot diagram. At a crossing, the overstrand might be called $a_k$ while the understrands are labeled with $a_i$ and $a_j$. A ``coloring" of these arcs can be viewed as assignments of the numbers $0$, $1$, or $2$ to our variables. This is not just an arbitrary choice of symbols: it turns out that the equation $2a_k-a_i-a_j=0$ holds mod $3$ if and only if the three ``colors" (or numbers) are all the same or all different. Recall that arithmetic mod $3$ means that any number $n$ is considered equivalent to any other number with the same remainder as $n$ when divided by $3$. For example, $1 \equiv 4 \equiv 7$ (mod $3$) and $-1 \equiv 2 \equiv 14$ (mod $3$). Another notation we can use to describe a number $n$ and its equivalents mod 3 is $\overline{n}$. The set of these equivalence classes is called $\mathbb{Z}/3\mathbb{Z}=\{\overline{0},\overline{1},\overline{2}\}$ and has a ring structure. This leads us to recall the following classical result:
\begin{theorem}\label{tricolorable-alternate}
A knot $K$ is {\bf tricolorable} if and only if for any diagram of $K$, the arcs of the diagram (labeled with the variables $a_1, a_2, a_3, ..., a_n$) can be assigned values from the ring $\mathbb{Z}/3\mathbb{Z}=\{\overline{0},\overline{1},\overline{2}\}$ (the integers mod $3$) such that at least two elements are used, and at every crossing, we have that $$2a_k - a_i - a_j \equiv 0 \pmod{3}$$ where $a_k$ is the overstrand label and $a_i$ and $a_j$ correspond to the understrands at the crossing.
\end{theorem}

Note that, in the theorem above, it is possible for $a_k$, $a_i$, and $a_j$ to not all be distinct, as is the case for a reducible crossing. This case does not pose a problem, since assigning the same ``color" to all three variables satisfies the equation.

\subsection{Knot {\em p}-colorability and the knot determinant}\label{sec:color}

Generalizing the notion of tricolorability to $p$-colorability is as simple as replacing ``$\mathbb{Z}/3\mathbb{Z}$" in Theorem \ref{tricolorable-alternate} with $\mathbb{Z}/p\mathbb{Z}$. For completeness, we will state the resulting definition: 

\begin{definition}\label{def:p-colorable}
A knot $K$ is $\bm{p}${\bf -colorable} if and only if for any diagram of $K$, the arcs of the diagram (labeled with the variables $a_1, a_2, a_3, \dots, a_n$) can be assigned values from the ring $\mathbb{Z}/p\mathbb{Z}=\{\overline{0},\overline{1},\overline{2},\dots,\overline{p-1}\}$ such that at least two elements are used, and at every crossing, we have that $$2a_k - a_i - a_j \equiv 0 \pmod{p}$$ where $a_k$ is the overstrand label and $a_i$ and $a_j$ correspond to the understrands at the crossing. 
\end{definition}

As an example, we can $5$-color the figure-eight knot as shown in Figure \ref{fig:colored-figure-eight}.\\

\begin{figure}[ht]
    \centering
    \includegraphics[width=0.35\textwidth]{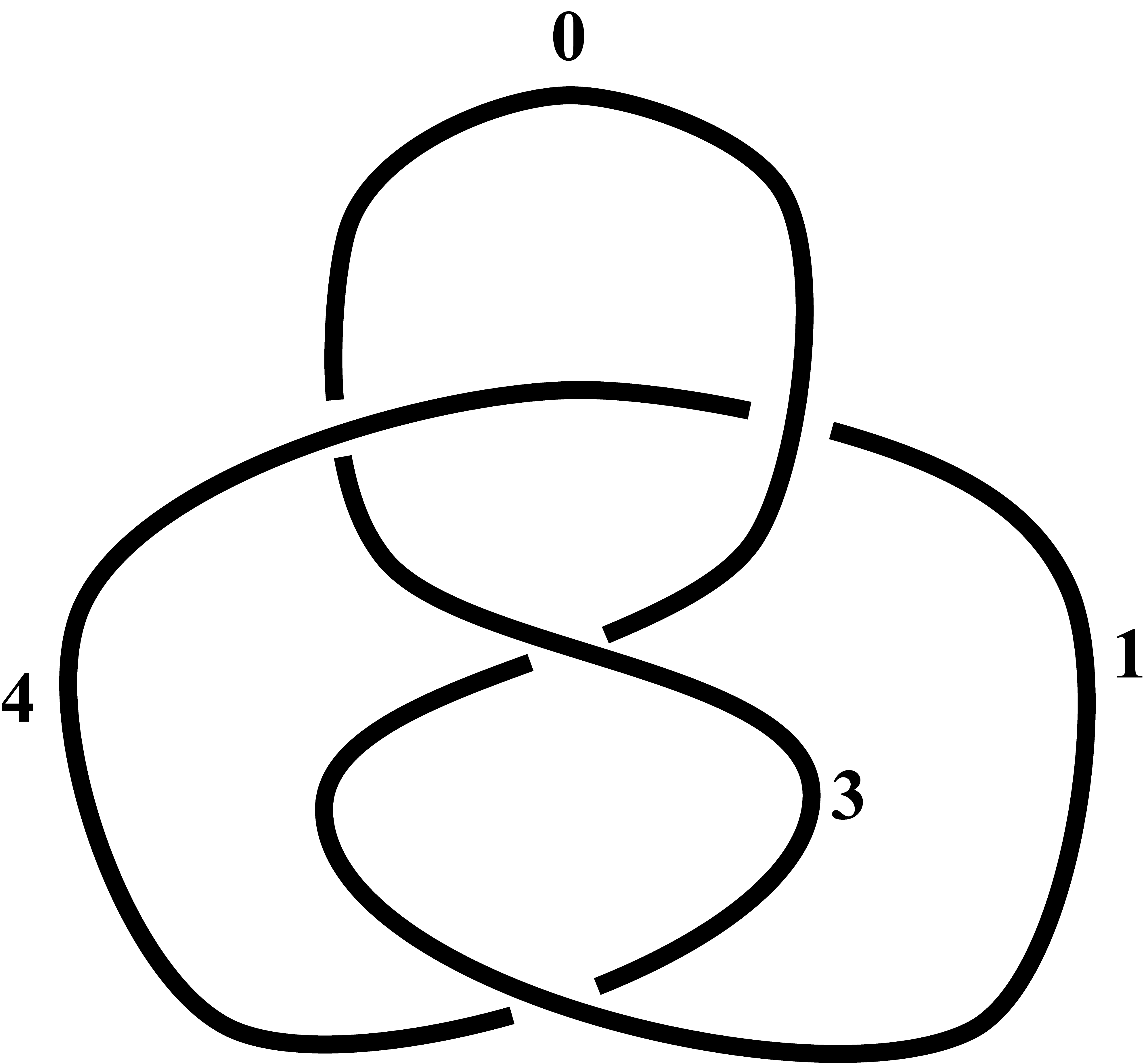}
    \caption{The figure-eight knot.}
    \label{fig:colored-figure-eight}
\end{figure}

Notice that if we index the crossings in a knot diagram using the integers $\{1, \dots, n\}$ and index the arcs in the diagram using the same set, we can create an $n \times n$ matrix $\mathcal{M}$ with elements in $\mathbb{Z}/p\mathbb{Z}$ that represents the system of linear equations generated by the $p$-colorability definition. In this matrix, each row corresponds to a crossing, and each column corresponds to an arc. Since there are always at least $p$ solutions (the trivial ones where we ``color" all arcs the same number), the determinant of this matrix will always be $0$. To get a more useful invariant, we take a first minor of $\mathcal{M}$ (that is, we strike a row and column from our matrix and take the determinant of the resulting submatrix) -- if it is $0$, the knot is $p$-colorable. Otherwise, it is not. This determinant is independent of both the diagram chosen to represent the knot and its labeling. Thus, we see a nice connection between colorability and linear algebra.\\

However, we would like to extend this connection by finding a way to test knots for all primes $p$ at once. To do this, we construct the above matrix as a matrix in $\mathbb{Z}$ and take a first minor to get an invariant called the {\bf determinant} of the knot. By our work in the previous paragraph, a knot is $p$-colorable if and only if $p$ divides the determinant of the knot (since then the minor goes to $0$ mod $p$).\\

A final note for this section: $p$-colorability is only defined for ordinary, double-crossing diagrams of knots, since our definition only allows for three arcs interacting at a single crossing. We are about to venture into a realm of nonstandard projections of knots to which the definitions above don't obviously apply. One of the main goals of this paper is to describe a process for determining colorability from these nonstandard knot representations. Before we can describe this process, we must describe the projections we're interested in studying.\\

\section{Petal projections, split petal projections, and their properties}

\subsection{Petal projections}

In \cite{adams2}, Adams {\em et al.} introduce the concept of a petal projection---a special type of diagram of a knot involving a single crossing, called an {\em \"{u}bercrossing diagram}.

\begin{definition}\label{petal}
A {\bf petal projection} is a knot projection with only one crossing---that has many strands of the knot passing through it---and no nested loops. Each strand passing through the crossing is assigned a number $1,2,..., n$ that indicates the height of the strand, where $1$ is assigned to the topmost strand and $n$ is assigned to the bottom strand. Here, $n$ is called the {\bf petal number} of the petal projection. A petal projection with petal number $n$ is also called an {\bf $n$-petal projection}.
\end{definition}

We see an example of a petal projection of a trefoil in Figure \ref{fig:trefoil-pp}. The trefoil is the only non-trivial knot that requires just $5$ petals to draw.\\

\begin{figure}
    \centering
    \includegraphics[width=0.25\textwidth]{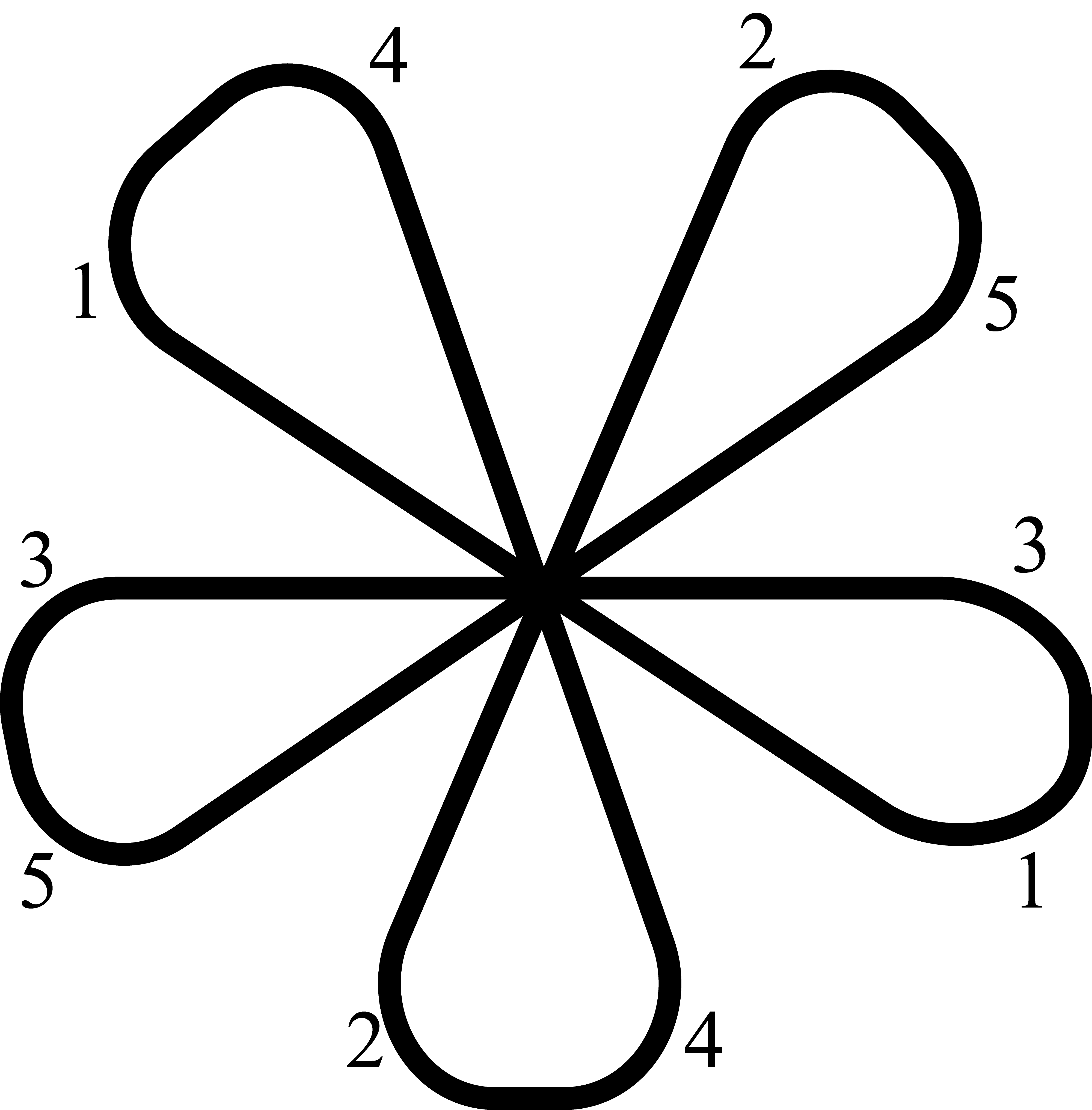}
    \caption{A petal projection of the trefoil}
    \label{fig:trefoil-pp}
\end{figure}

Petal projections are typically described by a permutation that describes the relative ordering of the strands, taken from the projection's labeling. There are many {\em petal permutations} that are equivalent. Since we can begin traveling around the petal projection at any petal, recording the heights of the strands we encounter along the way, we can cyclically permute the permutation that describes the projection. For example, the trefoil in Figure \ref{fig:trefoil-pp} is defined by the permutation $(1, 3, 5, 2, 4)$, but it could also be defined as $(3, 5, 2, 4, 1)$ or $(5, 2, 4, 1, 3)$. For this reason, we follow the convention in \cite{adams2} to always write petal permutations starting with the $1$.\\

We can also imagine turning the petal projection upside down, which doesn't change the knot type or the number of petals in the projection. This is equivalent to making the highest strand the lowest, the second highest the second lowest, and so on. Therefore, the trefoil could also be written: $(5, 3, 1, 4, 2)$ (and then shifted to $(1, 4, 2, 5, 3)$). Notice that while $1$ is at the start of both this representation and the original one, they are still two different petal permutations describing the same knot.\\

Furthermore, since we can always move the bottom strand around to the top of the pile of \"{u}bercrossing strands, we can start recording the order of the strands at any height. So, the permutation $(1, 6, 4, 2, 5, 7,3)$ could also be written as $(2,7,5,3,6,1,4)$ by adding 1 to each entry (mod 7), which can be rewritten as $(1,4, 2, 7, 5, 3, 6)$.\\

Another important property of petal permutations is that if any two consecutive numbers in the permutation are consecutive integers, we may pull the petal with the consecutive integers through the multi-crossing, merging three petals into one. This is possible because, as the two integers on the petal correspond to heights, the petal labeled with two consecutive integers occupies its own vertical slice of the multi-crossing, and therefore can be pulled through without interference. This is difficult to visualize without a diagram, so one is provided in Figure \ref{fig:simplify-pp}. Notice that the petal labeled with 3 and 4 is the one that is removed, so when the heights are relabeled, all of the heights greater than 4 are reduced by 2.\\

\begin{figure}
    \centering
    \includegraphics[width=0.60\textwidth]{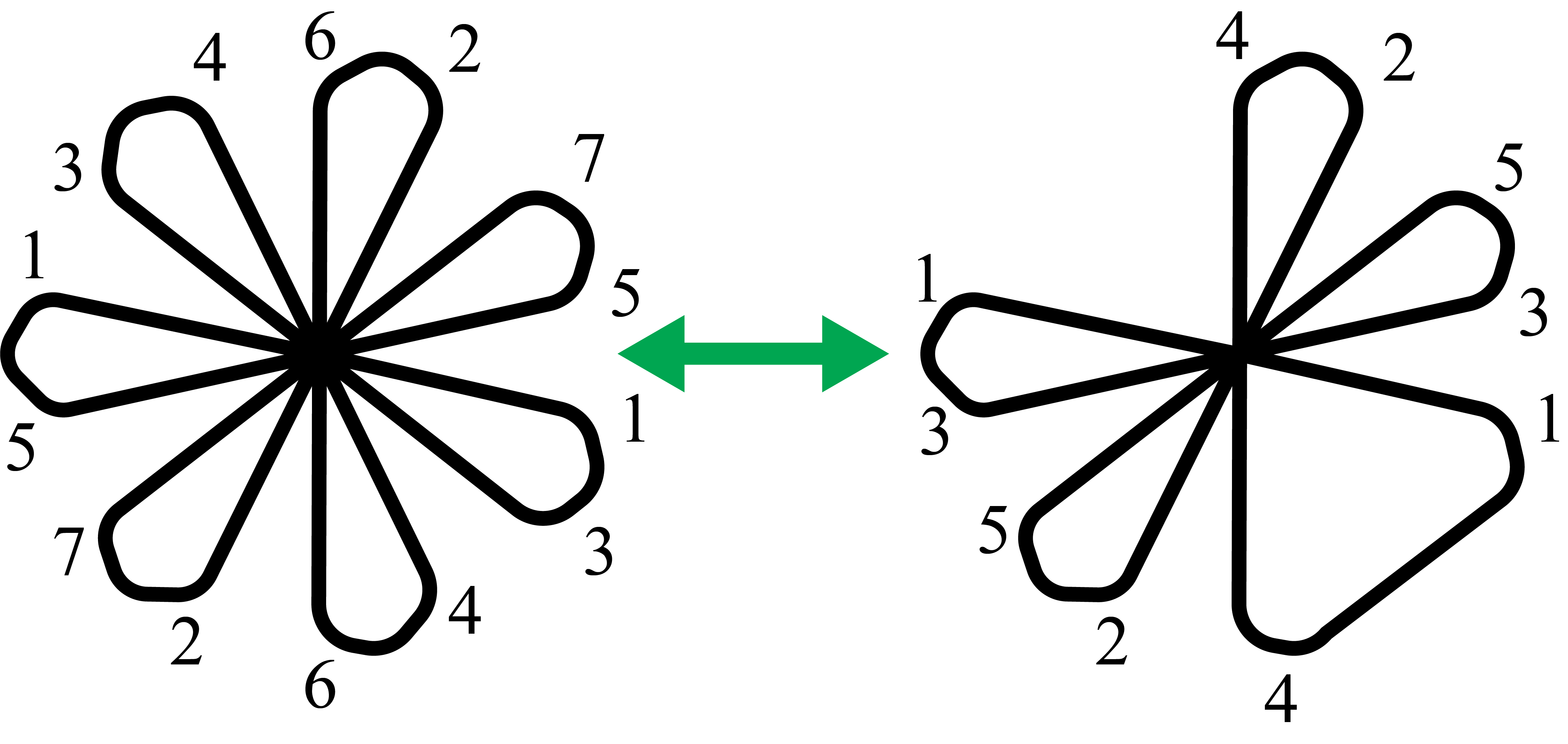}
    \caption{Simplifying a petal projection.}
    \label{fig:simplify-pp}
\end{figure}

This result is very interesting, as it significantly reduces the number of possible irreducible petal projections we need to consider. This alone, with the properties discussed above, shows that the trefoil is the only knot with petal number $5$. Indeed, there exists a more complete classification of when two petal permutations describe the same knot, given in the recent paper \textit{A Reidemeister type theorem for petal diagrams of knots} \cite{petal-reidemeister}. This is an exciting development, as it gives more credence to the theory that petal projections provide a new, useful language from the geometric to the purely algebraic.\\

The last property of petal projections that we observe is that all petal projections of knots must have an odd number of petals. This is a result of the fact that if we draw a petal projection with an even number of petals (say, $2n$), we  end up with an $n$-component link. Figure \ref{fig:2n-petals} provides an example.
\begin{figure}
    \centering
    \includegraphics[width = 0.50\textwidth]{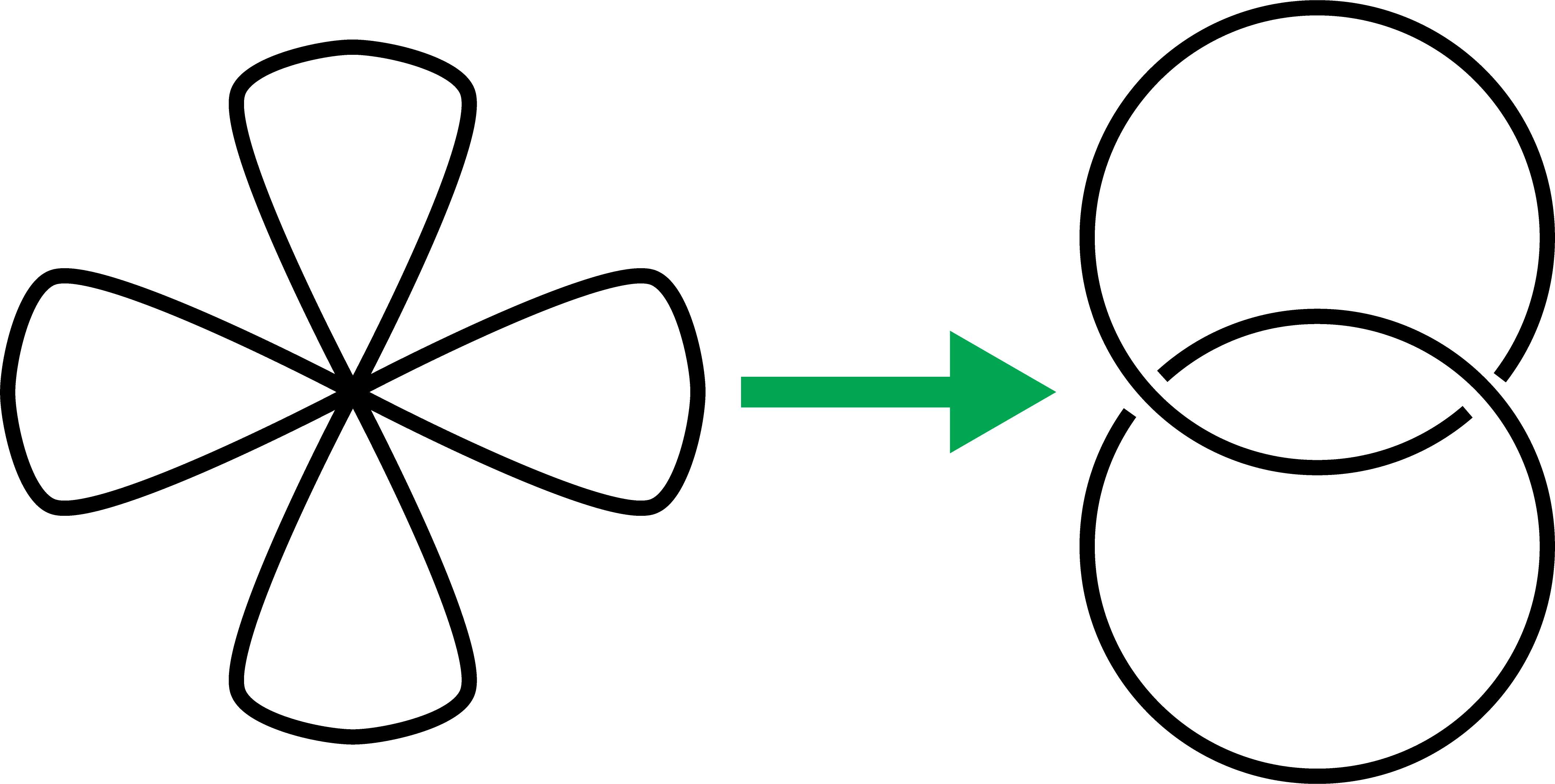}
    \caption{Any petal projection with an even number of petals splits into a (possibly linked) linked union of unknots.}
    \label{fig:2n-petals}
\end{figure}

\subsection{Split petal projections}

\begin{definition}
A {\bf split petal projection} is a redrawing of a petal projection in such a way that: 
\begin{itemize}
    \itemsep0em
    \item There are only double-crossings. 
    \item The split petal projection has the same symmetry as the original petal projection. 
    \item There are no reducible crossings.
    \item The split petal projection is uniquely identified by the same petal permutation. 
\end{itemize}
The notion of split petal projections briefly appears in \cite{adams2} under the name ``star patterns", but is not explored.
\end{definition}

For example, in Figure \ref{fig:trefoil-pp-spp}, we see a 5-petal projection of the trefoil knot, and its corresponding split petal projection.

\begin{figure}[ht]
    \centering
    \includegraphics[width=0.50\textwidth]{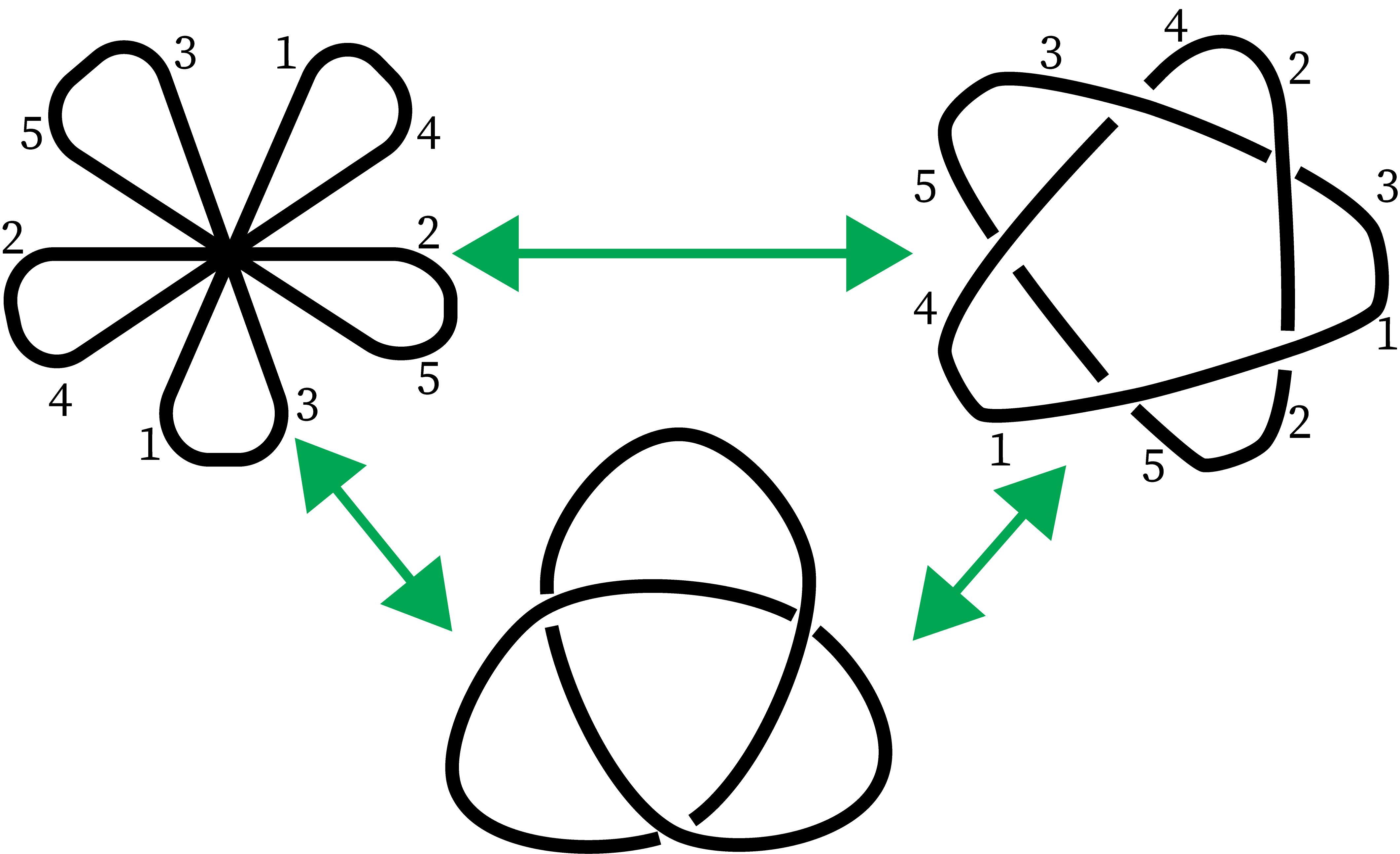}
    \caption{A petal projection and split petal projection of the trefoil.}
    \label{fig:trefoil-pp-spp}
\end{figure}

To see how to ``split'' a petal projection, it is helpful to work backwards---we instead visualize how a split petal projection can be reformed into a petal projection, although the process is invertible. \\

In Figure \ref{fig:petal-projection-repair}, we provide an example showing how to assemble a petal projection from a split petal projection. Here, we'll work with the knot shadow for simplicity. In the figure, we're using the example of a split petal projection with $7$ petals, but the process is analogous for any number of petals. Simply fold each strand (a portion of the knot extending from one of the outer-most points of the diagram to another) into the center, one at a time.

\begin{figure}[ht!]
    \centering
    \includegraphics[width=0.80\textwidth]{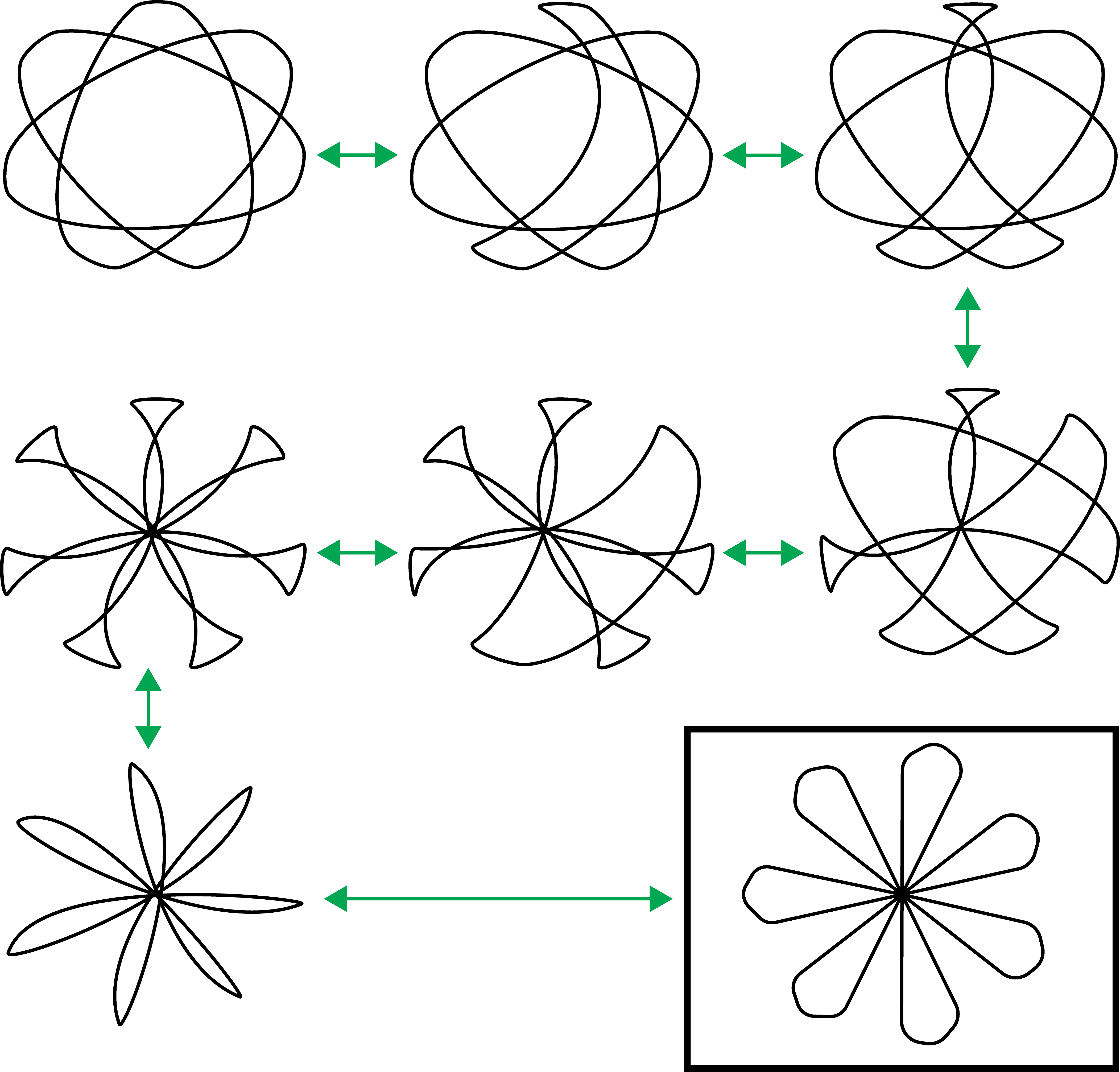}
    \caption{Repairing a split petal projection.}
    \label{fig:petal-projection-repair}
\end{figure}

\begin{definition}\label{spp-petal-number}
The {\bf petal number} of a split petal projection is the petal number of the petal projection that arises from the folding process described in Figure \ref{fig:petal-projection-repair}. 
\end{definition}

Split petal projections have some particularly nice properties. For example,
\begin{proposition}
A split petal projection with $p$ petals has $\frac{p(p-3)}{2}$ crossings and the same number of arcs.
\end{proposition}
\begin{proof}
Split petal projections can be obtained by perturbing all the arcs at central crossing and then untwisting all the resulting loops, as seen in Figure \ref{fig:petal-projection-repair} (starting at the end of the sequence and working backwards). Now, initially each of the $p$ strands going through the \"{u}bercrossing intersects each other strand exactly once. However, there is only one crossing because all of these intersections happen at a single point. In contrast, after the initial perturbation of the arcs off of the central intersection point, each of the $p$ strands still intersects each of the other $p-1$ strands, but now we only have double-crossings. This results in $\frac{p(p-1)}{2}$ crossings in total.  Then, the untwisting of all the loops (to avoid reducible crossings) reduces the number of crossings by $p$, giving $\frac{p(p-3)}{2}$ crossings, as desired. Finally, we note that, in any knot diagram with at least one crossing, the number of crossings equals the number of arcs.
\end{proof}

In a petal projection, a strand of the knot passing through the \"{u}bercrossing has a certain height, described by the petal permutation. The key fact about split petal projections is that the crossings ``inherit" these heights in a symmetric way, so that each strand height from the petal projection helps to determine the crossing information at $p-3$ crossings in the split petal projection.\\

One way to think about this is that a strand of the knot associated to a given height intersects the strands associated to each of the other heights exactly once when strands are initially perturbed to avoid the central \"{u}bercrossing point, resulting in $p-1$ crossings associated to a given height. The two crossings at the ends of the strand are removed when loops are untwisted, resulting in a total of $p-3$ crossings that are labeled by a given height.\\

There is another way to come to this answer. The key is to notice that the number and arrangement of crossings is independent of the heights; another way to say this is that the \textit{projection} of a split petal projection is independent of its petal permutation. Therefore, any given height will appear at the same number of crossings, by symmetry. This gives the following equation: 
\begin{equation*}
    \frac{\text{\# of crossings} \times \text{\# of heights at each crossing}}{\text{\# of heights}} = \text{\# of crossings per height}
\end{equation*}
Filling in the requisite information, we again find that a given height appears at $p-3$ crossings: 
\begin{equation*}
    \frac{\left(p(p-3)/2\right) \cdot 2}{p} = p-3
\end{equation*}
This style of reasoning will be critical in the next section, when we develop a completely algebraic formula for the Gauss code of a split petal projection.\\






\section{Gauss Codes}
A {\em Gauss code} for a knot is a sequence of integers that describes the order in which you encounter crossings in a knot diagram when traveling along the knot. 
To generate a Gauss code from a knot diagram, we first number the crossings in the diagram $1$ through $n$ (where $n$ is the total number of crossings) in any order. Then, starting at any point on the knot and choosing a direction to travel, we traverse the knot by following the strand as it weaves through the knot. Whenever we pass through a crossing, we record either the crossing's number (if we travel ``over" the crossing) or the crossing number's negative (if we travel ``under" the crossing).\\

\begin{figure}
    \centering
    \includegraphics[width=0.35\textwidth]{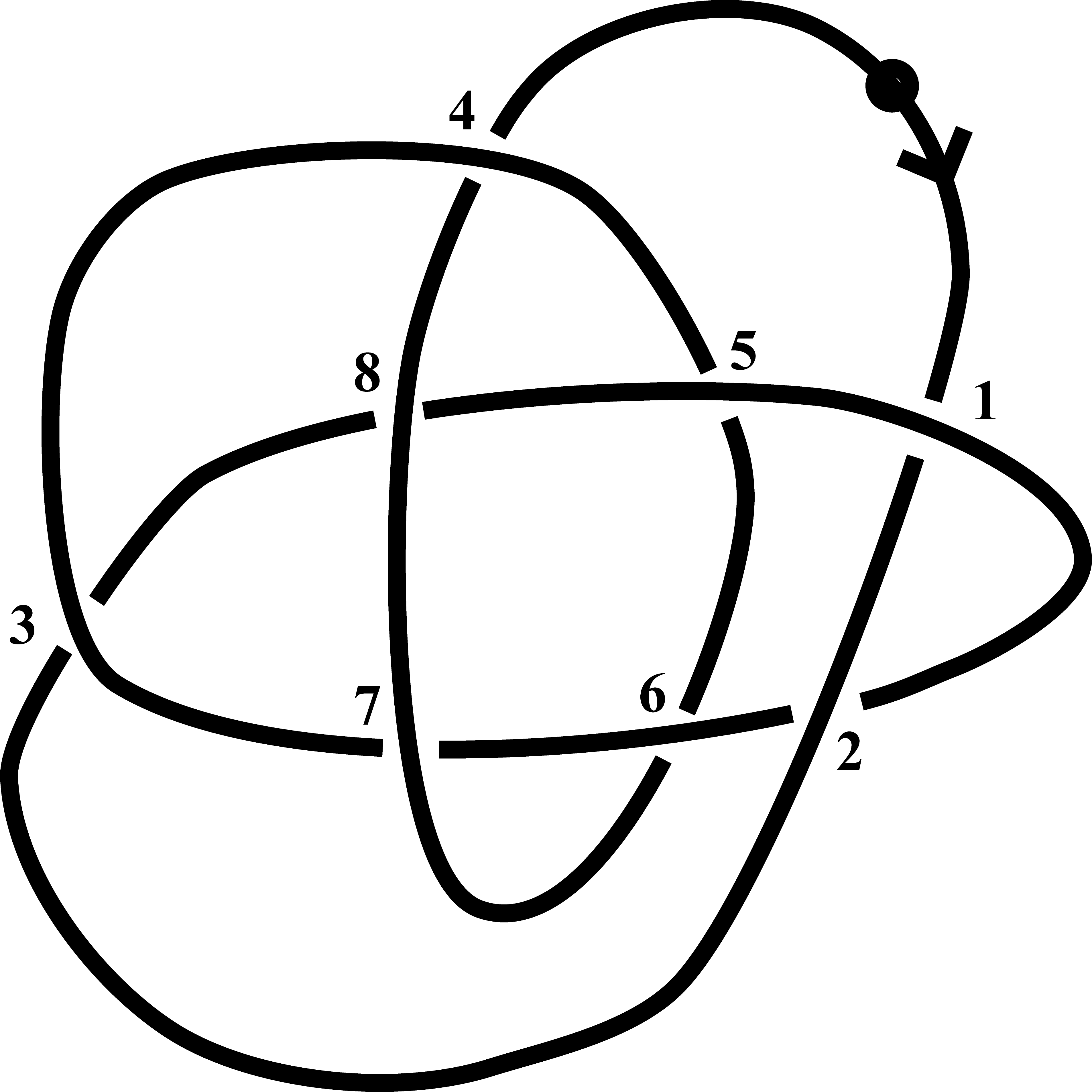}
    \caption{A Gauss code of the $8_{19}$ knot: $-1,2,-3,-8,5,1,-2,6,-7,3,4,-5,-6,7,8,-4$.}
    \label{fig:8-19-gauss}
\end{figure}

Visually, calculating the Gauss code of a simple knot is easy. See Figure \ref{fig:8-19-gauss} for an example. While each knot may be represented by many different Gauss codes (generated by different diagrams, numberings, base points, etc.), a given Gauss code determines a knot type uniquely (up to mirror image).\\

Now, our main goal is to evaluate the knot determinant for general split petal projections, so we need an algorithm  to generate the Gauss code for any split petal projection from a petal permutation. In short, we'd like to find a function with a petal permutation as an input and a Gauss code---preferably one that inherits certain properties from the symmetry of the petal projection---as the output.

\subsection{Gauss Codes for Split Petal Projections}
Before we begin describing an algorithm to generate Gauss codes for split petal projections, let's generate a simpler, {\em unsigned} Gauss code. The advantage of beginning with an unsigned Gauss code (that doesn't record over-under information about crossings) is that we don't need to know in advance the information about relative heights of strands that will be provided by the petal permutation. We simply need to know the petal number.\\

Let's consider a na\"ive attempt to create an algorithm to find an unsigned Gauss code for a given petal number. We will use the example of the $7$-split petal projection with petal permutation $(1, 3, 5, 2, 7, 4, 6)$, the figure-eight knot. One way to start might be to label each crossing in the order in which we first encounter it. This starts simply, giving an unsigned Gauss code that begins $1,2,3,\dots$ for the first few numbers. However, this method comes with its own problems; while the Gauss code starts simple, it quickly gets complicated, as shown in Figure \ref{fig:gauss-naive-7pp}.\\

\begin{figure}
    \centering
    \includegraphics[width=0.38\textwidth]{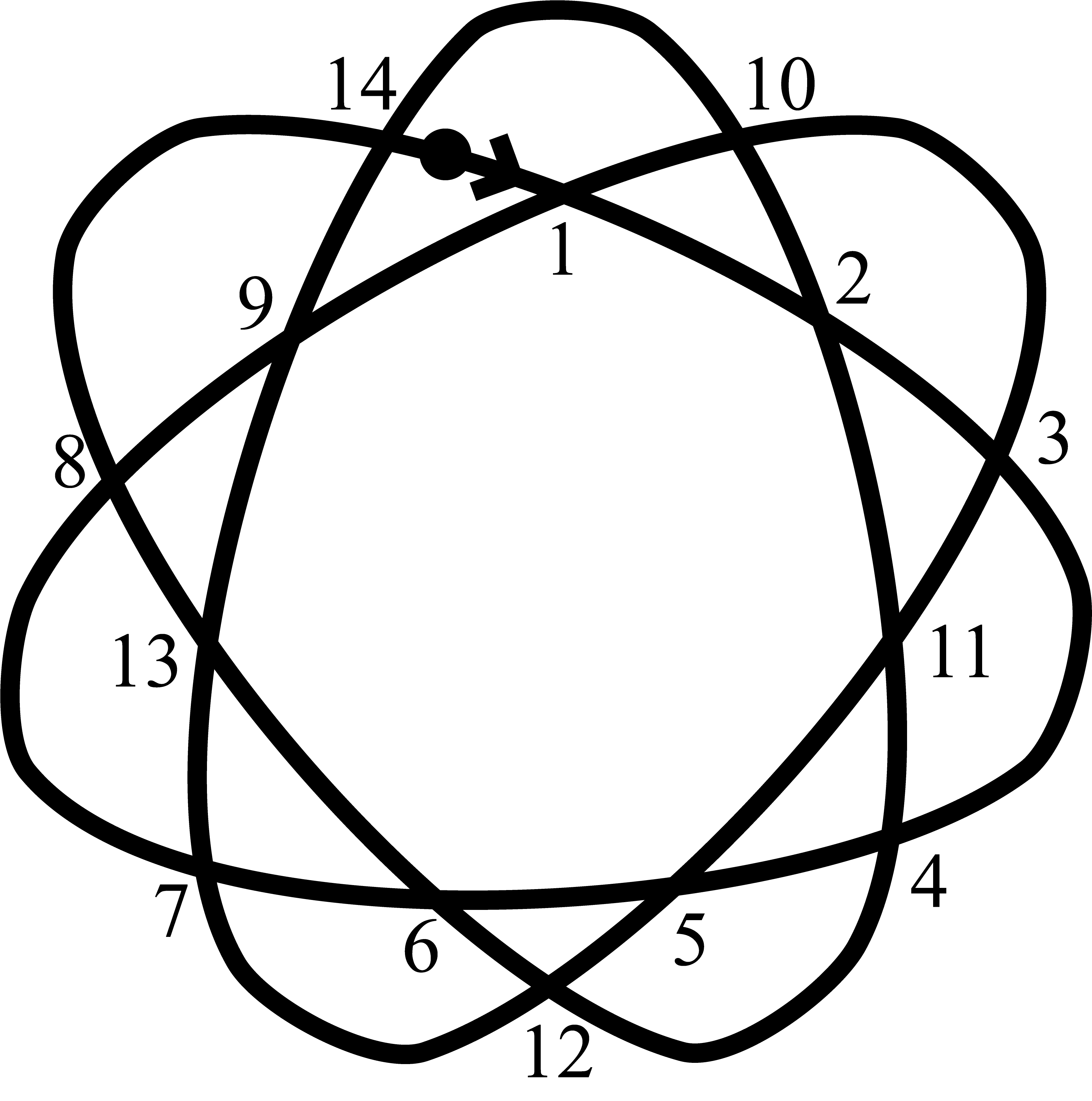}
    \caption{A naive attempt to generate the unsigned Gauss code of a split petal projection: \\ $1,2,3,4,5,6,7,8,9,1,10,3,11,5,12,7,13,9,14,10,2,11,4,12,6,13,8,14$.}
    \label{fig:gauss-naive-7pp}
\end{figure}
A better way to generate a useful unsigned Gauss code starts by considering the structure of the split petal projection. Notice that in general, there are $\frac{p-3}{2}$ ``layers" of crossings, each with $p$ symmetrically distributed crossings. These layers appear in concentric circles, with the first layer (Layer `0') being all the crossings that touch the central region, and each next layer being the next concentric circle out. For instance, when the petal number $p = 9$, there are $\frac{9-3}{2} = 3$ layers, as shown in in Figure \ref{fig:stevedore-layers}. Layer 0 consists of all crossings at the vertices of the central nonagon. Layer 1 consists of the nine crossings lying just outside the nonagon, and Layer 2 contains the remaining nine crossings along the boundary of the exterior region of the diagram. Another example, this one for $p = 7$, can be seen in Figure \ref{fig:general-gauss-7pp}.\\

\begin{figure}
    \centering
    \includegraphics[width=0.40\textwidth]{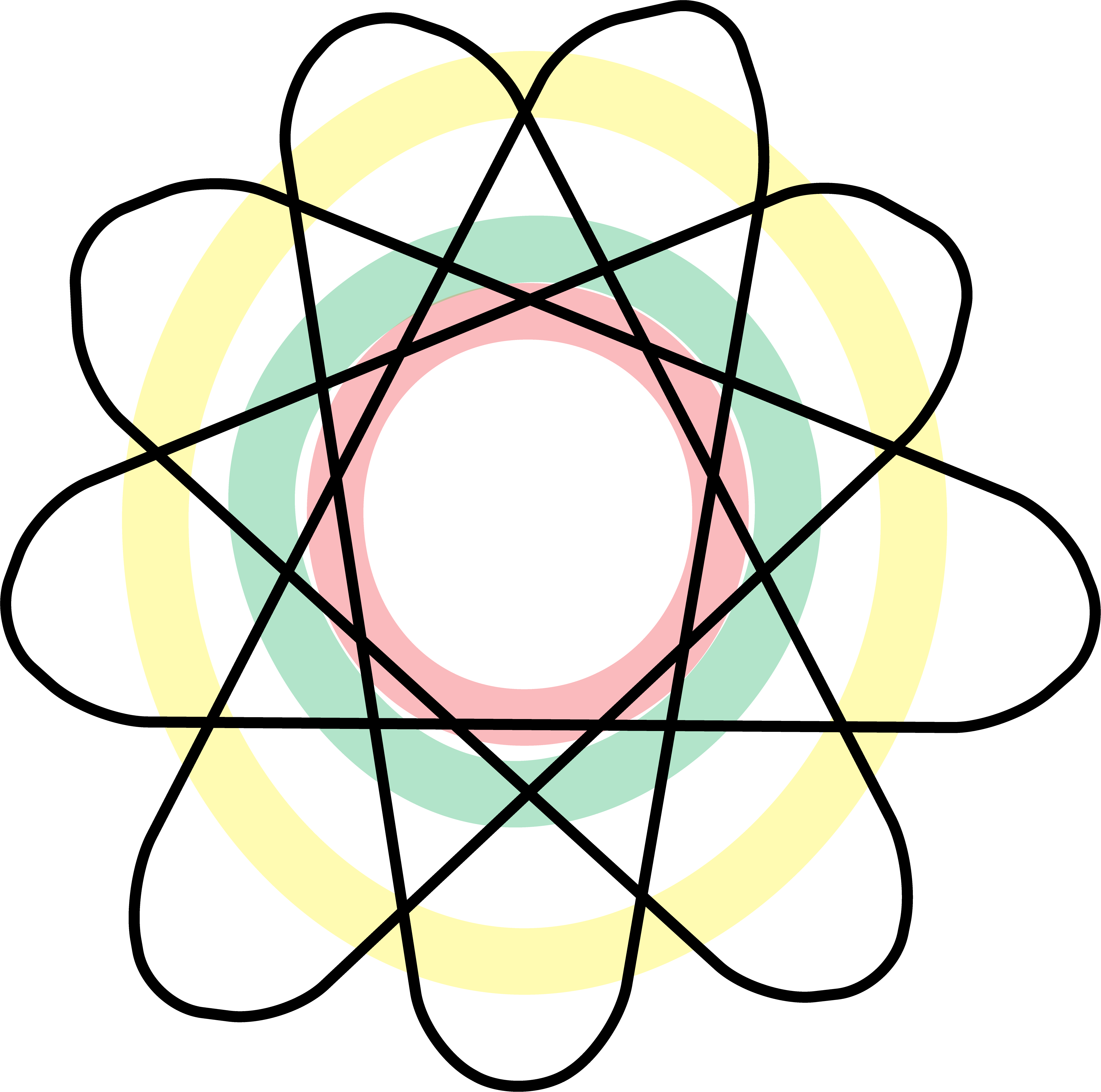}
    \caption{Layer $0$ is the set of crossings lying in the innermost (red) circle, Layer $1$ is the set of crossings in the middle (green) circle, and Layer $2$ is the set of crossings in the outermost (yellow) circle.}
    \label{fig:stevedore-layers}
\end{figure}

Now, instead of labeling the crossings directly, we begin by labeling each crossing with an ordered pair $(L,n)$, where $L$ denotes the layer that the crossing is in, and $n$ is the ``index" of the crossing \textit{in the particular layer}. We determine the index of each crossing by letting the first crossing we encounter in each layer be index $1$, and number the rest in the layer clockwise. Numbering the crossings this way more clearly shows a general pattern in the unsigned Gauss code. This can be verified with the example of the $7$-petal projection, shown in Figure \ref{fig:general-gauss-7pp}. \\

\begin{figure}
    \centering
    \includegraphics[width=0.40\textwidth]{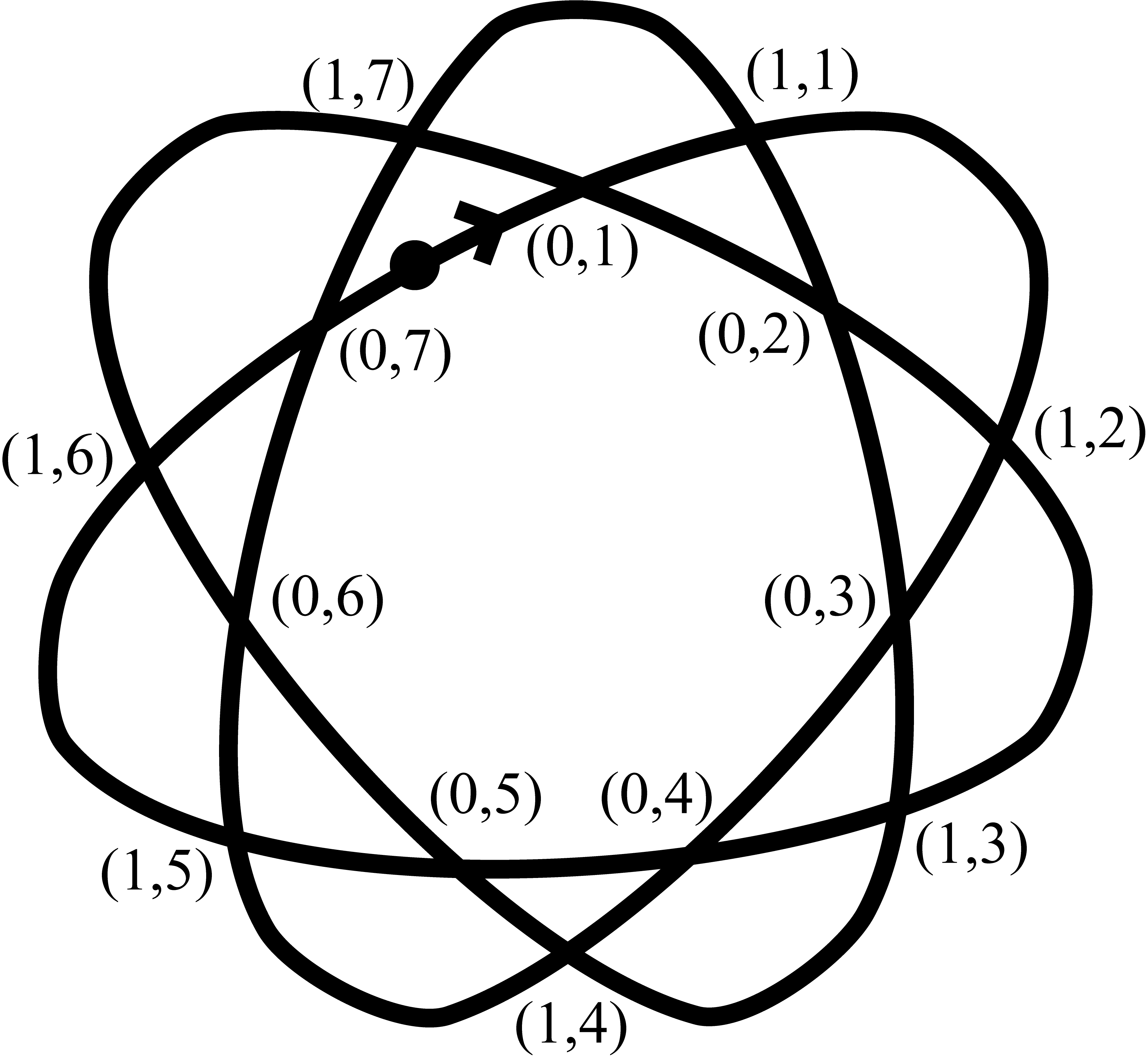}
    \caption{The unsigned Gauss code of a $7$-split petal projection is:\\$(0,1),(1,1),(1,2),(0,3),(0,4),(1,4),(1,5),(0,6),(0,7),(1,7),(1,1),(0,2),(0,3),(1,3),$\\
    $(1,4),(0,5),(0,6),(1,6),(1,7),(0,1),(0,2),(1,2),(1,3),(0,4),(0,5),(1,5),(1,6),(0,7)$}
    \label{fig:general-gauss-7pp}
\end{figure}

As you can see, the first entries in the ordered pairs appear in a repeating palindrome: $0, 1, 1, 0$. For larger knots, this palindrome will grow exactly as expected, first to $0,1,2,2,1,0$, then to $0,1,2,3,3,2,1,0$, and so on. This result generalizes into Definition \ref{def:L}. This palindrome repeats in blocks of length $p-3$. From the block perspective, the second number's pattern is just as predictable. In a $7$-petal projection, each block is $4$ numbers long, and the numbers come in the form $n, n, n+1, n+2$. Note that we reduce each number in the block, replacing it with the smallest \textit{positive} integer congruent to it mod $p$. In a $9$-petal projection, the blocks have length $6$, and the numbers come in the form $n, n, n, n+1, n+2, n+3$. Again, this result will generalize into Definitions \ref{def:a} and \ref{def:n_k}.\\

To demonstrate our algorithm, we consider the Stevedore knot, a $9$-petal knot with petal permutation $(1, 3, 5, 2, 8, 4, 6, 9, 7)$. See Figure \ref{fig:general-gauss-9pp}. \\

\begin{figure}
    \centering
    \includegraphics[width=0.4\textwidth]{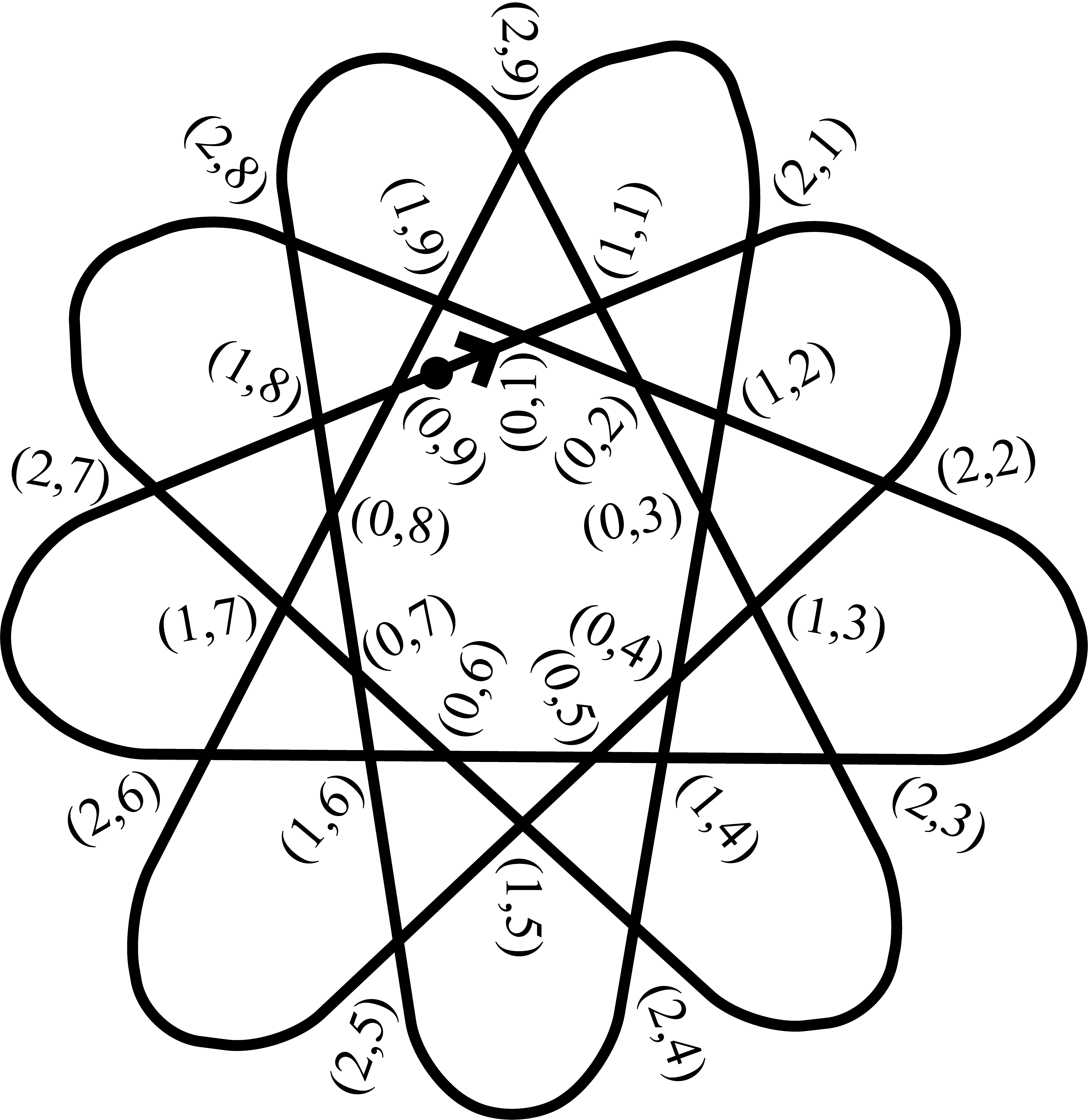}
    \caption{The unsigned Gauss code of a $9$-split petal projection is: $(0,1),(1,1),(2,1),(2,2),(1,3),(0,4),(0,5),(1,5),(2,5),$
    $(2,6),(1,7),(0,8),(0,9),(1,9),(2,9),(2,1),(1,2),(0,3),$
    $(0,4),(1,4),(2,4),(2,5),(1,6),(0,7),(0,8),(1,8),(2,8),$
    $(2,9),(1,1),(0,2),(0,3),(1,3),(2,3),(2,4),(1,5),(0,6),$
    $(0,7),(1,7),(2,7),(2,8),(1,9),(0,1),(0,2),(1,2),(2,2),$
    $(2,3),(1,4),(0,5),(0,6),(1,6),(2,6),(2,7),(1,8),(0,9)$.}
    \label{fig:general-gauss-9pp}
\end{figure}

If we can explicitly generate this pattern for any size petal permutation, we will have effectively generated an unsigned Gauss code for each petal number, since we can simply map $(L,n) \to L \cdot p + n$ (where $p$ is the petal number). This map is a bijection from the set of pairs $(L,n)$ to the integers $1$ through $\frac{p(p-3)}{2}$. We know this because $n$ ranges between $1$ and $p$ (since there are $p$ crossings in each layer), and $L$ varies from $0$ to $\frac{p-3}{2}-1$, since there are $\frac{p-3}{2}$ layers. This bijection inspires Proposition \ref{prop:unsigned}, which completes the algorithm. In Figure \ref{fig:immediate-gauss-9pp}, we give the example for the Stevedore knot.\\

\begin{figure}
    \centering
    \includegraphics[width=0.4\textwidth]{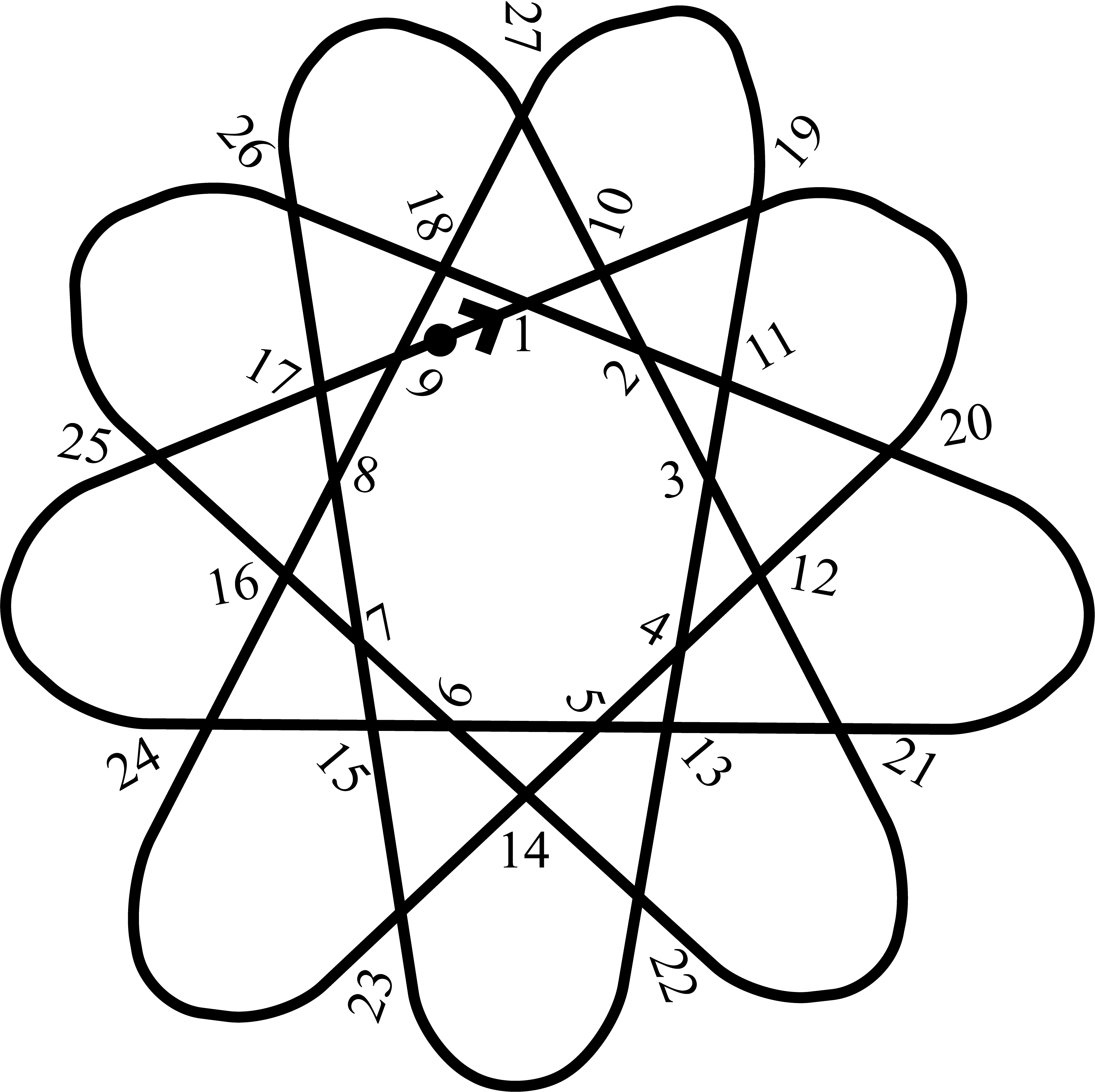}
    \caption{The unsigned Gauss code of a $9$-split petal projection is: \\
    $1,10,19,20,12,4,5,14,23,24,16,8,9,18,27,19,11,3,$\\ $4,13,22,23,15,7,8,17,26,27,10,2,3,12,21,22,14,6,$\\ $7,16,25,26,18,1,2,11,20,21,13,5,6,15,24,25,17,9$.}
    \label{fig:immediate-gauss-9pp}
\end{figure}
Now, our goal is to take our unsigned Gauss code and use the petal permutation to determine where to insert the negative signs that denote underpasses. To achieve this, we use the fact that a strand in the petal projection associated with a given height is mapped to a strand in the split petal projection passing through $p-3$ crossings. This allows us to make a table mapping a crossing label from the unsigned Gauss code to the height of the over- or under-strand at that crossing. After making this table, we can determine if the strand associated to a number in the unsigned Gauss code is going over or under at the crossing to make the signed Gauss code. Again, we will use the example of the Stevedore knot, which has a $9$-split petal projection with petal permutation (1, 3, 5, 2, 8, 4, 6, 9, 7):\\

\begin{table}[]
    \centering
    \begin{tabular}{||c c c c c c c c c c c c c c c c c c c||} 
     \hline
     \textbf{Index} & {\bf 0} & 1 & 2 & 3 & 4 & 5 & 6 & 7 & 8 & 9 & 10 & 11 & 12 & 13 & 14 & 15 & 16 & 17\\ 
     \hline
     \textbf{Gen. Gauss Code} & {\bf 1} & 10 & 19 & 20 & 12 & 4 & 5 & 14 & 23 & 24 & 16 & 8 & 9 & 18 & 27 & 19 & 11 & 3\\ 
     \hline
     \textbf{P.P. Height} & {\bf 1} & 1 & 1 & 1 & 1 & 1 & 3 & 3 & 3 & 3 & 3 & 3 & 5 & 5 & 5 & 5 & 5 & 5\\
     \hline 
     \hline
     \textbf{Index} & 18 & 19 & 20 & 21 & 22 & 23 & 24 & 25 & 26 & 27 & 28 & 29 & 30 & 31 & 32 & 33 & 34 & 35\\ 
     \hline
     \textbf{Gen. Gauss Code} & 4 & 13 & 22 & 23 & 15 & 7 & 8 & 17 & 26 & 27 & 10 & 2 & 3 & 12 & 21 & 22 & 14 & 6\\ 
     \hline
     \textbf{P.P. Height} & 2 & 2 & 2 & 2 & 2 & 2 & 8 & 8 & 8 & 8 & 8 & 8 & 4 & 4 & 4 & 4 & 4 & 4\\ 
     \hline
     \hline
     \textbf{Index} & 36 & 37 & 38 & 39 & 40 & {\bf 41} & 42 & 43 & 44 & 45 & 46 & 47 & 48 & 49 & 50 & 51 & 52 & 53\\
     \hline
     \textbf{Gen. Gauss Code} & 7 & 16 & 25 & 26 & 18 & {\bf 1} & 2 & 11 & 20 & 21 & 13 & 5 & 6 & 15 & 24 & 25 & 17 & 9\\ 
     \hline
     \textbf{P.P. Height} & 6 & 6 & 6 & 6 & 6 & {\bf 6} & 9 & 9 & 9 & 9 & 9 & 9 & 7 & 7 & 7 & 7 & 7 & 7\\ 
     \hline
    \end{tabular}
    
    \vspace{.2in}
    \caption{The unsigned Gauss code of a 9-split petal projection (which may resolve, for example, into the Stevedore knot when we sign the entries).}
    \label{tab:stevedore-unsigned-example}
\end{table}

From the table, we can finally create the signed Gauss code for the Stevedore knot as follows. Note, for instance, that the first ``1" in the Gauss code is positive because the height of the associated instance of ``1" in the table is 1, while the height of the second instance of ``1" in the table is 6. Since lower numbered arcs are taken to be above higher numbered arcs, this tells us that the first pass through crossing 1 should be an overpass. \\

\begin{center}
\fbox{%
    \parbox{400pt}{%
    \textbf{Stevedore Knot Gauss Code:}\\
    $\bm{1}, 10,19,20,12,4,5,14,-23,24,16,8,9,18,27,-19,11,-3,-4,13,22,23,15,7,-8,-17,-26,$ $-27,-10,2,3,-12,21,-22,-14,6,-7,-16,25,26,-18,\bm{-1},-2,-11,-20,-21,-13,-5,-6,$
    $-15,-24,-25,17,-9$
    }%
}
\end{center}
$ $\newline

We can use this to draw the split petal projection with crossing information clearly marked, as in Figure \ref{fig:signed-gauss-9pp}.When we create a general algorithm to generate the Gauss code of a split petal projection from a petal permutation, breaking it down into these individual steps allows for much easier computation. At this point, we are ready to explicitly describe how to determine the unsigned Gauss code for a given petal number. To begin, we will establish several notational definitions.

\begin{figure}
    \centering
    \includegraphics[width=0.45\textwidth]{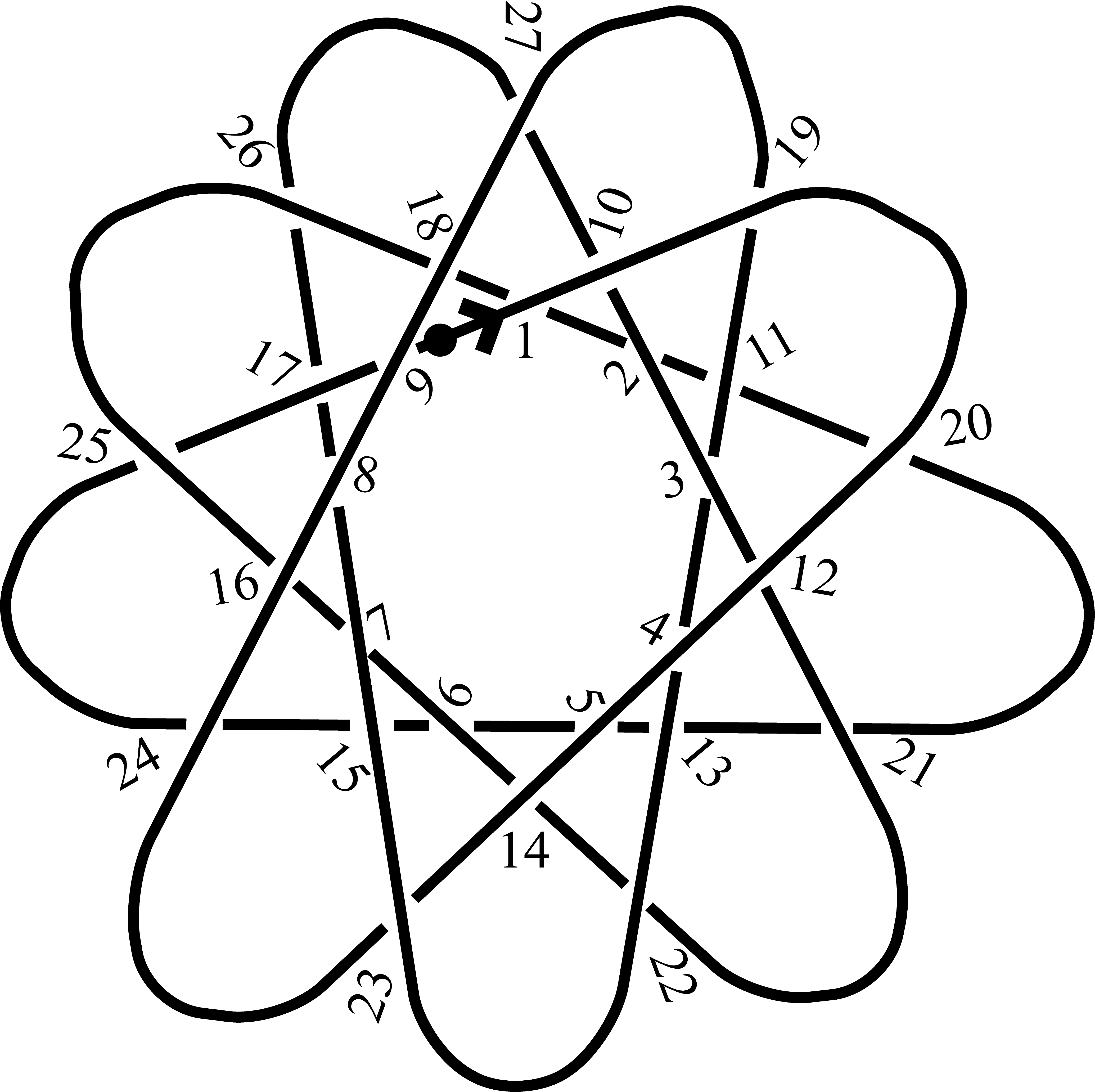}
    \caption{The Stevedore Knot's Split Petal Projection.}
    \label{fig:signed-gauss-9pp}
\end{figure}

\begin{definition}\label{def:r_n(k)}
Let $\bm{r_n(k)}$ be the smallest non-negative integer congruent to $k$ mod $n$, i.e., the remainder when $k$ is divided by $n$. For example, $r_2(5) = 1$, and $r_3(9) = 0$. 
\end{definition}

\begin{definition}\label{def:L}
Let $p$ be a petal number. Then $\bm{L^p}$ is the repeating palindromic sequence $$0,1, 2, \dots, \frac{p-5}{2}, \frac{p-5}{2}, \dots ,2,1,0, 0,1, 2, \dots, \frac{p-5}{2}, \frac{p-5}{2}, \dots 2, 1,0, \dots$$
        which we will denote by $L_0^p, L_1^p, \dots$. When the $p$ is understood, we may suppress the superscripts and denote our sequence by $L_0, L_1, \dots$ . 
        
        Note that the periodicity of $L^p$ is $p-3=\frac{p-5}{2}+\frac{p-5}{2}+2$.
\end{definition}

This definition encodes the ``layers" visited as we follow the knot, as discussed earlier.

\begin{definition}\label{def:a}
Let $p$ be a petal number. Then $\bm{a^p}$ is the sequence $$\underbrace{0,\dots,0}_{\frac{p-3}{2} \text{ times}},1,2,\dots, \frac{p-3}{2}, \underbrace{0,\dots,0}_{\frac{p-3}{2} \text{ times}},1,2,\dots, \frac{p-3}{2}, \dots $$
        denoted by $a^p_0, a^p_1, \dots$ . As before, when $p$ is understood, we may drop the superscripts and write $a_0, a_1, \dots$ . 
        
        Just as with the sequence $L^p$, the periodicity of $a^p$ is $p-3$ for all $p$.
\end{definition}

\begin{definition}\label{def:n_k}
Define $\bm{n_k^p}$ to be the quantity given by
$$r_p\left(\frac{p-1}{2} \cdot \left \lfloor \frac{k}{p-3} \right \rfloor + a^p_k\right) + 1.$$
\end{definition}

Similarly, this definition encodes what we called the ``index in a particular layer." Now, we are ready to describe a general formula for the unsigned Gauss code of a petal projection with $p$ petals.

\begin{proposition}\label{prop:unsigned}
An unsigned Gauss code for a split petal projection with $p$ petals is given by the sequence $$c_0, \dots c_{p(p-3)-1}$$
where $c_k = p \cdot L_k + n_k$. We call this specific Gauss code the {\bf petal unsigned Gauss code}.
\end{proposition}


\begin{definition}\label{def:unsigned}
In a (signed or unsigned) Gauss code $C$ corresponding to a petal projection with $p$ petals, the {\bf index} $\bm{k}$ {\bf petal block} corresponds to the $p-3$ crossings $c_{k(p-3)}, \dots, c_{(k+1)(p-3)-1}$. Notice that a petal projection with $p$ petals will have exactly $p$ petal blocks, each corresponding to a height in the petal permutation.
\end{definition}

Consider, for example, Table \ref{tab:stevedore-unsigned-example}. In this table, there are a total of $9$ ``blocks" which all have the same height in the petal permutation, and all of these blocks have the same length (namely $9 - 3 = 6$). To help illustrate the previous definitions, we will again compute the unsigned Gauss code of a $9$-split petal projection in Table \ref{tab:stevedore-unsigned-example-from-algorithm}. Compare this to Table \ref{tab:stevedore-unsigned-example} -- note that the row $c_k$ in Table \ref{tab:stevedore-unsigned-example-from-algorithm} is identical to the unsigned Gauss code computed in Table \ref{tab:stevedore-unsigned-example}.

\begin{table}[h]
    \centering
    \begin{tabular}{||c c c c c c c c c c c c c c c c c c c||} 
     \hline
     \textbf{Index (${\bm k}$)} & 0 & 1 & 2 & 3 & 4 & 5 & 6 & 7 & 8 & 9 & 10 & 11 & 12 & 13 & 14 & 15 & 16 & 17\\ 
     \hline
     $\bm{L^p}$ (see Def. \ref{def:L}) & 0 & 1 & 2 & 2 & 1 & 0 & 0 & 1 & 2 & 2 & 1 & 0 & 0 & 1 & 2 & 2 & 1 & 0\\
     \hline
     $\bm{a^p}$ (see Def. \ref{def:a}) & 0 & 0 & 0 & 1 & 2 & 3 & 0 & 0 & 0 & 1 & 2 & 3 & 0 & 0 & 0 & 1 & 2 & 3\\
     \hline
     $\bm{n^p}$ (see Def. \ref{def:n_k}) & 1 & 1 & 1 & 2 & 3 & 4 & 5 & 5 & 5 & 6 & 7 & 8 & 9 & 9 & 9 & 1 & 2 & 3\\
     \hline
     ${\bm c_k}$ (see Prop. \ref{prop:unsigned}) & 1 & 10 & 19 & 20 & 12 & 4 & 5 & 14 & 23 & 24 & 16 & 8 & 9 & 18 & 27 & 19 & 11 & 3\\
     \hline \hline
     \textbf{Index (${\bm k}$)} & 18 & 19 & 20 & 21 & 22 & 23 & 24 & 25 & 26 & 27 & 28 & 29 & 30 & 31 & 32 & 33 & 34 & 35\\ 
     \hline
     $\bm{L^p}$ (see Def. \ref{def:L}) & 0 & 1 & 2 & 2 & 1 & 0 & 0 & 1 & 2 & 2 & 1 & 0 & 0 & 1 & 2 & 2 & 1 & 0\\
     \hline
     $\bm{a^p}$ (see Def. \ref{def:a}) & 0 & 0 & 0 & 1 & 2 & 3 & 0 & 0 & 0 & 1 & 2 & 3 & 0 & 0 & 0 & 1 & 2 & 3\\
     \hline
     $\bm{n^p}$ (see Def. \ref{def:n_k}) & 4 & 4 & 4 & 5 & 6 & 7 & 8 & 8 & 8 & 9 & 1 & 2 & 3 & 3 & 3 & 4 & 5 & 6\\
     \hline
     ${\bm c_k}$ (see Prop. \ref{prop:unsigned}) & 4 & 13 & 22 & 23 & 15 & 7 & 8 & 17 & 26 & 27 & 10 & 2 & 3 & 12 & 21 & 22 & 14 & 6\\
     \hline\hline
     \textbf{Index (${\bm k}$)} & 36 & 37 & 38 & 39 & 40 & 41 & 42 & 43 & 44 & 45 & 46 & 47 & 48 & 49 & 50 & 51 & 52 & 53\\
     \hline
     $\bm{L^p}$ (see Def. \ref{def:L}) & 0 & 1 & 2 & 2 & 1 & 0 & 0 & 1 & 2 & 2 & 1 & 0 & 0 & 1 & 2 & 2 & 1 & 0\\
     \hline
     $\bm{a^p}$ (see Def. \ref{def:a}) & 0 & 0 & 0 & 1 & 2 & 3 & 0 & 0 & 0 & 1 & 2 & 3 & 0 & 0 & 0 & 1 & 2 & 3\\
     \hline
     $\bm{n^p}$ (see Def. \ref{def:n_k}) & 7 & 7 & 7 & 8 & 9 & 1 & 2 & 2 & 2 & 3 & 4 & 5 & 6 & 6 & 6 & 7 & 8 & 9\\
     \hline
     ${\bm c_k}$ (see Prop. \ref{prop:unsigned}) & 7 & 16 & 25 & 26 & 18 & 1 & 2 & 11 & 20 & 21 & 13 & 5 & 6 & 15 & 24 & 25 & 17 & 9\\
     \hline
    \end{tabular}
    \vspace{.2in}
    
    \caption{The unsigned Gauss code $c_0,\dots,c_{p(p-3)-1}$ of a 9-split petal projection, computed directly using the definitions created above.}
    \label{tab:stevedore-unsigned-example-from-algorithm}
\end{table}

\subsection{The Signed Gauss Code} 
From here, we seek to produce a \textit{signed} Gauss code, which depends not only on the number of petals, but also on our knot's petal permutation. Converting to the signed Gauss code from the unsigned Gauss code is not difficult: it simply involves negating each $c_k$ that is associated with an understrand at a crossing.\\

To do this, we need to find the heights of the two strands that make up a crossing. For this, it will be helpful to have the following definition, which we will draw on in the following sections: 

\begin{definition}
Suppose $k$ is an index in the petal unsigned Gauss code $c_0,\dots,c_{p(p-3)-1}$. Then define $k'$ to be the unique index of the petal unsigned Gauss code such that $k \neq k'$ but $c_k = c_{k'}$. 
\end{definition}

Note that $k$ and $k'$ correspond to one another, meaning that $(k')'=k$. For example, if we are working with the Stevedore knot, and take $k = 0$, then $k' = 41$, since $c_0 = c_{41} = 1$, and similarly if we take $k = 41$, then $k' = 0$ for the same reason. Then, our table allows us to see that the height of $c_{41}$ is $6$ but the height of $c_0$ is $1$, so $c_0$ is the overstrand and stays positive in the signed Gauss code, while $c_{41}$'s sign is flipped. Note that in general, it suffices to identify the \textit{petal block} that $c_{k'}$ is in, since each block corresponds to a single number in the petal permutation. For instance, for each $j\in\{36, 37, 38, 39, 40, 41\}$, $c_j$ in the Stevedore unsigned Gauss code has height 6. So in fact, we don't need to know that $k'=41$ specifically. We simply need to know that $36 \leq k' < 42$.\\

Thus, we seek a simple expression for $\lfloor \frac{k'}{p-3} \rfloor$, the index of the petal block that $c_{k'}$ is in (between $0$ and $p-1$, inclusive). If we use the previous example of $k = 0$ and $k' = 41$, we can find the petal block that $k'$ is in (the 7th block, with index $6$---coincidentally the same number as the height associated to this index) by plugging in $k'=41$ and $p=9$ to get $\lfloor \frac{41}{9-3} \rfloor = 6$. To get the most out of identifying heights using petal blocks rather than $k'$ itself, we seek a more general expression for block indices that doesn't rely on a brute-force search to find $k'$. 
\begin{lemma}\label{lem:0}Given that $r_n(k)$ denotes the smallest non-negative integer congruent to $k$ mod $n$, then 
$$\left\lfloor \frac{k'}{p-3} \right\rfloor = r_p\left(\left\lfloor\frac{k}{p-3}\right\rfloor + 2(a_{k'} - a_k)\right).$$
where, as before, $k'$ is the unique integer such that $c_k = c_{k'}$ but $k' \neq k$.
\end{lemma}

\begin{proof}
First note that $c_{k'} = c_k$ implies that both $L_{k'} = L_k$ and $n_{k'} = n_k$, by Proposition \ref{prop:unsigned}. We may use the definition of $n_k$, Definition \ref{def:n_k}, to get the following.

$$\frac{p-1}{2} \cdot \left\lfloor \frac{k}{p-3} \right\rfloor + a_k \equiv \frac{p-1}{2} \cdot \left\lfloor \frac{k'}{p-3} \right\rfloor + a_{k'} \mod p$$

We also use the fact that $(p-1)$ is congruent to $-1$ mod $p$ and multiply through by $-2$.

$$\left\lfloor \frac{k}{p-3} \right\rfloor - 2a_k \equiv \left\lfloor \frac{k'}{p-3} \right\rfloor - 2a_{k'} \mod p$$

Finally, we rearrange and note that both $\left\lfloor \frac{k}{p-3} \right\rfloor$ and $\left\lfloor \frac{k'}{p-3} \right\rfloor$ are bounded by $0$ and $p-1$, giving us the desired result: 

$$\left\lfloor \frac{k'}{p-3} \right\rfloor = r_p\left(\left\lfloor\frac{k}{p-3}\right\rfloor + 2(a_{k'} - a_k)\right).$$
\end{proof}

This will prove to be a useful formula, but we still lack a crucial piece: it is difficult to determine what $a_{k'}$ will be without knowing $k'$, rendering this formula in its current form less useful. Therefore, we seek an expression for $a_{k'}-a_k$ that is dependent only on $k$. To find this, note that both $a_n = a_{r_{p-3}(n)}$ and $L_n = L_{r_{p-3}(n)}$ for all $n$. These follow from Definitions $\ref{def:L}$ and $\ref{def:a}$, the definitions of $a$ and $L$, since both sequences were defined to be periodic with period $p-3$.

\begin{lemma}\label{lem:1} Given that $r_n(k)$ denotes the smallest non-negative integer congruent to $k$ mod $n$, then $$r_{p-3}(k) + r_{p-3}(k') = p-4$$ for each $k$ and the corresponding $k'$. 
\end{lemma}

This lemma can be intuitively understood as follows. Recall that $k$ and $k'$ denote the indices of a given crossing as it appears twice in an unsigned Gauss code.  Because the two indices correspond to the two distinct strand heights associated to the same crossing, they must be in different petal blocks. What this lemma tells us is that they appear in their respective blocks in complementary positions. For instance, if $k$ appears as the first index in its block, then $k'$ appears as the last index in its block. If $k$ appears as the second index in its block, then $k'$ appears as the second-to-last index in its block, and so forth. Thus, the sum of the remainders of the two indices when divided by block length equals $p-4$. 

\begin{proof} Consider that $c_k = c_{k'}$ implies that $L_k = L_{k'}$. Because of the palindromic nature of $L_k$, this implies that either $r_{p-3}(k) = r_{p-3}(k')$ or $r_{p-3}(k) + r_{p-3}(k') = p-4$.\\

To prove that we are in the second case, we assume that the first equality holds and derive a contradiction. If $r_{p-3}(k) = r_{p-3}(k')$, then $a_k = a_{k'}$, by the periodicity of the $a$ sequence. Combined with the fact that $n_k = n_{k'}$, this forces $k$ and $k'$ to be in the same petal block. In this case, the equality of $r_{p-3}(k) = r_{p-3}(k')$ implies that $k=k'$, contradicting the definition of $k'$. Thus, $r_{p-3}(k) + r_{p-3}(k') = p-4$, as desired.\end{proof}

\begin{lemma}\label{lem:2}

The quantity $a_{k'}-a_k$ is given by
\[a_{k'}-a_k = \begin{cases} 
\frac{p - 3}{2} - r_{p-3}(k) & \text{if }\; r_{p-3}(k) < \frac{p-3}{2}, \\
\frac{p-3}{2} - r_{p-3}(k) - 1 & \text{otherwise.} \end{cases} \]
\end{lemma}

\begin{proof}
By Lemma \ref{lem:1}, $k$ and $k'$ appear at complementary points in the sequence $a$. Thus, either $a_k=0$ or $a_{k'}=0$. This allows us to characterize $a_{k'}-a_k$ as follows.
\begin{itemize}
    \item \textbf{Case 1:} Suppose that $a_k = 0$. Then we are in the first half of $a$'s cycle, so $r_{p-3}(k) < \frac{p-3}{2}$. In this case, by Definition \ref{def:a} and Lemma \ref{lem:1}, $$a_{k'} - a_k = a_{k'} = r_{p-3}(k') + 1 - \frac{p-3}{2} = p - 3 - r_{p-3}(k) - \frac{p-3}{2} = \frac{p-3}{2} - r_{p-3}(k).$$
    \item \textbf{Case 2:} Suppose that $a_{k'} = 0$, so $r_{p-3}(k) \geq \frac{p-3}{2}$. In this case, by Definition \ref{def:a}, $$a_{k'} - a_k = -a_k = \frac{p-3}{2} - r_{p-3}(k) - 1.$$
\end{itemize}
These two cases combine to give us our desired piecewise function.
\end{proof}
Thus, we have the following definition which will enable us to identify the petal blocks of corresponding crossing labels in the unsigned Gauss code, given in Definition \ref{def:block}.
\begin{definition}\label{def:prep-block}
Define the following function.
\[ \bm{d(k)} = 2(a_{k'} - a_k) = \begin{cases} 
p - 3 - 2r_{p-3}(k) & r_{p-3}(k) < \frac{p-3}{2}, \\
p - 5 - 2r_{p-3}(k) & \textrm{otherwise}. \end{cases} \]
\end{definition}

\begin{definition}\label{def:block}
The petal block containing $k$ is given by $\bm{b(k)} = \left\lfloor \frac{k}{p-3} \right\rfloor$, and the petal block containing $k'$ is given by  $\bm{b'(k)} = r_p\left(\left\lfloor\frac{k}{p-3}\right\rfloor + d_k\right)$. 
\end{definition}

Using the machinery developed above, we now can determine the {\em signed} Gauss code from a petal permutation as follows.

\begin{theorem}
Let $P = h_0, h_1, \dots h_{p-1}$ be a petal permutation describing a given knot in terms of a $p$-petal projection. Let $C = c_0, \dots c_{p(p-3)-1}$ denote the corresponding petal unsigned Gauss code, as defined in Definition \ref{def:unsigned}.  Then a signed Gauss code corresponding to $P$ is given by the following: $C' = c'_0, \dots c'_{p(p-3)-1}$, where $c'_i = c_i$ if $h_{b(k)} < h_{b'(k)}$ and  $c'_i = -c_i$ otherwise. We call this specific Gauss code the {\bf petal signed Gauss code} for $P$. 
\end{theorem}
Consider, as an illustration of this theorem, Table \ref{tab:signed-gauss-code-from-algorithm}. Table \ref{tab:signed-gauss-code-from-algorithm} lists the unsigned Gauss code, the $k'$ corresponding to each $k$ (which depends only on the unsigned Gauss code by definition), the heights derived from the petal permutation of the Stevedore knot, and uses them to sign the unsigned Gauss code of a $9$-split petal projection corresponding to the Stevedore knot along the lines of Theorem $7$. 

\begin{table}[h]
    \centering
    \begin{tabular}{||c c c c c c c c c c c c c c c c c c c||} 
     \hline
     \textbf{Index (${\bm k}$)} & {\bf 0} & 1 & 2 & 3 & 4 & 5 & 6 & 7 & 8 & 9 & 10 & 11 & 12 & 13 & 14 & 15 & 16 & 17\\ 
     \hline
     ${\bm c_k}$ & {\bf 1} & 10 & 19 & 20 & 12 & 4 & 5 & 14 & 23 & 24 & 16 & 8 & 9 & 18 & 27 & 19 & 11 & 3\\ 
     \hline
     ${\bm k'}$ & {\bf 41} & 28 & 15 & 44 & 31 & 18 & 47 & 34 & 21 & 50 & 37 & 24 & 53 & 40 & 27 & 2 & 43 & 30\\
     \hline
     \textbf{P.P. Height} & {\bf 1} & 1 & 1 & 1 & 1 & 1 & 3 & 3 & 3 & 3 & 3 & 3 & 5 & 5 & 5 & 5 & 5 & 5\\
     \hline
     ${\bm c'_k}$ & {\bf 1} & 10 & 19 & 20 & 12 & 4 & 5 & 14 & -23 & 24 & 16 & 8 & 9 & 18 & 27 & -19 & 11 & -3\\ 
     \hline \hline
     \textbf{Index (${\bm k}$)} & 18 & 19 & 20 & 21 & 22 & 23 & 24 & 25 & 26 & 27 & 28 & 29 & 30 & 31 & 32 & 33 & 34 & 35\\ 
     \hline
     ${\bm c_k}$ & 4 & 13 & 22 & 23 & 15 & 7 & 8 & 17 & 26 & 27 & 10 & 2 & 3 & 12 & 21 & 22 & 14 & 6\\ 
     \hline
     ${\bm k'}$ & 5 & 46 & 33 & 8 & 49 & 36 & 11 & 52 & 39 & 14 & 1 & 42 & 17 & 4 & 45 & 20 & 7 & 48\\ 
     \hline
     \textbf{P.P. Height} & 2 & 2 & 2 & 2 & 2 & 2 & 8 & 8 & 8 & 8 & 8 & 8 & 4 & 4 & 4 & 4 & 4 & 4\\
     \hline
     ${\bm c'_k}$ & -4 & 13 & 22 & 23 & 15 & 7 & -8 & -17 & -26 & -27 & -10 & 2 & 3 & -12 & 21 & -22 & -14 & 6\\ 
     \hline\hline
     \textbf{Index (${\bm k}$)} & 36 & 37 & 38 & 39 & 40 & {\bf 41} & 42 & 43 & 44 & 45 & 46 & 47 & 48 & 49 & 50 & 51 & 52 & 53\\
     \hline
     ${\bm c_k}$ & 7 & 16 & 25 & 26 & 18 & {\bf 1} & 2 & 11 & 20 & 21 & 13 & 5 & 6 & 15 & 24 & 25 & 17 & 9\\ 
     \hline
     ${\bm k'}$ & 23 & 10 & 51 & 26 & 13 & {\bf 0} & 29 & 16 & 3 & 32 & 19 & 6 & 35 & 22 & 9 & 38 & 25 & 12\\ 
     \hline
     \textbf{P.P. Height} & 6 & 6 & 6 & 6 & 6 & {\bf 6} & 9 & 9 & 9 & 9 & 9 & 9 & 7 & 7 & 7 & 7 & 7 & 7\\ 
     \hline
     ${\bm c'_k}$ & -7 & -16 & 25 & 26 & -18 & {\bf 1} & -2 & -11 & -20 & -21 & -13 & -5 & -6 & -15 & -24 & -25 & 17 & -9\\ 
     \hline
    \end{tabular}
    \vspace{.2in}
    \caption{The signed Gauss code $c'_0,\dots,c'_{p(p-3)-1}$ of the Stevedore knot, calculated directly using the algorithm outlined above.}\label{tab:signed-gauss-code-from-algorithm}
\end{table}

We can now describe how this particular signed Gauss code for a split petal projection can help us to compute a knot's determinant from its petal permutation. We'll implement an algorithm for computing the knot determinant of split petal projections in Python. We'll test the results of our program against the those results about petal projections of all prime knots with fewer than $10$ crossings given in \cite{adams2}. We'll also test it against non-trivial knots of various sizes, from $5$-split petal projections up to $50$-split petal projections. These will appear in our appendices and will include run times over multiple trials.

\section{An algorithm to determine the colorability of a split petal projection}

\subsection{Description of the algorithm}

To determine the colorability of a knot from one of its petal permutations, $P$, we begin by labeling the arcs in the corresponding split petal projection of the knot with elements of $\mathbb{Z}/p\mathbb{Z}$. We then form a matrix from these labelings, as described Section \ref{sec:color}. Using linear algebra, we can determine the dimension of the linear system of equations' solution space (which represents the $p$-colorability of the knot). Since the trivial solutions form a space of dimension $1$, the knot is $p$-colorable if the solution space has dimension at least $2$. Let us now consider how exactly to construct this matrix, $\mathcal{M}_P$.\\

For concreteness, we will first explore the specific case of the Stevedore knot. We use the signed Gauss code to number the arcs. Recall that an arc begins and ends at an understrand -- thus arcs begin and end at negative numbers in the signed Gauss code, as follows: 
\begin{center}
\fbox{%
    \parbox{400pt}{%
    \textbf{Stevedore Knot Gauss Code:}\\
    $\bm{1}, 10,19,20,12,4,5,14,-23,24,16,8,9,18,27,-19,11,-3,-4,13,22,23,15,7,-8,-17,-26,$ $-27,-10,2,3,-12,21,-22,-14,6,-7,-16,25,26,-18,\bm{-1},-2,-11,-20,-21,-13,-5,-6,$
    $-15,-24,-25,17,-9$
    }%
}
\end{center}
\begin{center}
\fbox{%
    \parbox{400pt}{%
    \textbf{The Arc Start/End Points in the Stevedore Knot's Gauss Code:}\\
    $(-23,-19), (-19,-3), (-3,-4), (-8,-17), (-17,-26), (-26, -27), (-27, -10),(-10, -12), $ $(-12, -22), (-22, -14), (-14, -7), (-7, -18), \bm{(-18, -1)}, \bm{(-1, -2)}, (-11, -20), (-20, -21),$
    $(-21, -13), (-13, -5), (-5, -6), (-6, -15), (-15, -24),(-24, -25), (-25, -9), \bm{(-9, -23)}$
    }%
}
\end{center}

We will examine a single row of this matrix, corresponding to the crossing labeled $1$. Each of the arcs is represented by one of the columns in the matrix, numbered in the order shown above. We will place a $-1$ in the column associated with the two arcs which are understrands at crossing $1$, and a $2$ in the column representing the overstrands.

\[
\textrm{row }1\textrm{ of }\mathcal{M}_P = \begin{mpmatrix}
    0 & 0 & 0 & 0 & 0 & 0 & 0 & 0 & 0 & 0 & 0 & 0 & -1 & -1 & 0 & 0 & 0 & 0 & 0 & 0 & 0 & 0 & 0 & 0 & 0 & 0 & 2
\end{mpmatrix}
\]

The bold arcs are the three which involve crossing $1$. Notice that the two arcs that form the understrands have crossing $1$ as their start or end point. However, the arc that passes over crossing $1$ does not appear explicitly. Since one can check that $1$ is between $-9$ and $-23$ in the signed Gauss code, the arc $(-9,-23)$ is the arc that passes over crossing $1$.\\

At this point, for completeness, we will state the process we for finding the Gauss code's associated matrix in full generality: 
\begin{enumerate}
    \itemsep0em
    \item We start with a petal signed Gauss code of length $p(p-3)$, and define $\mathcal{M}$ to be a $\frac{p(p-3)}{2} \times \frac{p(p-3)}{2}$ matrix comprised entirely of $0$s.
    \item We split the Gauss code into arcs by ending the previous arc and starting a new one whenever we encounter a negative number, just as we did earlier with the Stevedore knot. 
    \item We include \textit{both endpoints} of the arc in our ``arc-sets" (which is why we say ``split" instead of ``partition"). \item We label the arcs we've created with the numbers $1$ to $\frac{p(p-3)}{2}$. For example, in the case of the Stevedore knot, the arc $(-23,-19)$ is Arc $1$, the arc $(-19,-3)$ is Arc $2$, etc.
    \item For each number $n \in \{1, \dots, \frac{p(p-3)}{2}\}$, we determine which arc is the overstrand at crossing $n$ (say Arc $i$) and which arcs are the understrands at crossing $n$ (say Arcs $j$ and $j+1$). 
    \item We set the $n$th row, $i$th column element to $2$, and the $n$th row, $j$th and $j+1$th column elements to $-1$, just we did when creating the first row of $\mathcal{M_P}$ for the Stevedore knot.
\end{enumerate}

Now, we take the first minor of this matrix (striking out the first row and column), treating the elements as members of $\mathbb{Z}$ (instead of how they were originally defined in $\mathbb{Z}/p\mathbb{Z}$), and see if this determinant is divisible by particular $p$. This approach allows us to test the $p$-colorability of the knot for all $p$ at once. Furthermore, we recall the classical result that the number of ways in which the knot is $p$-colorable is equal to the product of all $p$ in the prime factorization of this resulting minor. 
\begin{definition}
Let $\ord_p(n)$, where $p$ and $n$ are positive integers, be the largest power of $p$ that divides $n$. In other words, $k = \ord_p(n)$ is the unique nonnegative integer such that $p^k \mid n$ but $p^{k+1} \nmid n$. 
\end{definition}
The following is a well-known result \cite{livingston}.

\begin{theorem}
Let $M$ be the first minor of a coloring matrix for a knot $K$. This minor, called the {\em determinant} of the knot is an invariant. Furthermore, $K$ is colorable using $p$-colors in $p^{\ord_p(|M|)+1}$ ways. Since every knot has $p$ trivial colorings (one for each color), there are $p^{\ord_p(|M|)+1}-p$ nontrivial $p$-colorings of the knot. 
\end{theorem}

We now turn our attention to directly implementing the algorithm for determining a knot's determinant from one of its petal permutations in Python.

\subsection{Implementing the algorithm in Python.}
The first step in implementing the algorithm is to create a library of functions: one for determining the unsigned petal Gauss code from the petal number of the split petal projection, one for deriving the (signed) petal Gauss code from the unsigned code and the petal permutation, one for generating the petal Gauss code matrix from the petal Gauss code, and one for evaluating the knot determinant from this matrix. Figure 19 visually shows the way these functions fit together. \\

\begin{figure}[h]
\begin{center}
    \includegraphics[width=0.45\textwidth]{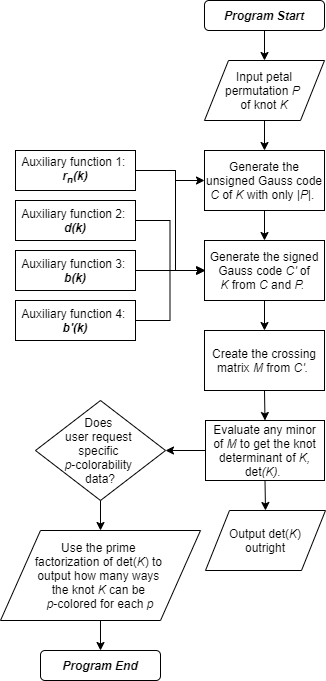}
\end{center}
    \caption{A visualization of the algorithm.}
    \label{fig:alg}
\end{figure}

How these functions are written can vary. Ours are written for readability first and speed second. We have included our full code (with documentation) in Appendix \ref{app:code}. We compute the determinants of all prime knots with fewer than 10 crossings from their petal permutations, given in \cite{adams1}, with runtime in Appendix \ref{app:prime}. We were also interested in discovering how our algorithm performs as we increase the size of the petal projection. Current bounds on the relationship between crossing number and petal number will mean that our experiments will deal with knots of very large size, pushing the code to its limits. This is further explored in Appendix \ref{app:limits}. Finally, we will do some experimental testing of the distribution of $p$-colorability of random knots using randomly selected permutations in Appendix \ref{app:dist}. 

\section{Conclusion}
Currently, the fastest determinant algorithms have complexity $\approx O(n^{2.373})$ and $n$ is $O(p^2)$, where $p$ is the petal number of the knot. This means that our algorithm should have complexity $\approx O(p^{4.746})$, since everything else we do has strictly less complexity. We do not know how this is related to crossing number, the standard measure of complexity for knots. However, Colin Adams et. al \cite{adams2} have proven that the crossing number, $c(K)$, is bounded above by $\frac{p(K)^2 - 2p(K) - 3}{4}$. Indeed, torus knots of the form $T_{r, r+1}$ realize the inequality, showing that this bound is tight. This is good news: it suggests that, at least for some torus knots, our algorithm can have sub-cubic complexity with respect to crossing number.\\

There are still more interesting questions we can ask in this setting. For instance, 
\begin{enumerate}
    \item Is there a faster (less computationally complex) way to compute the knot determinant than evaluating the first minor of the petal Gauss code matrix that takes further advantage of the symmetry of petal projections?
    \item Can one replicate this work using a wider class of objects (than elements of $\mathbb{Z}/p\mathbb{Z}$) to label the arcs of the split petal projection? For instance, can quandles be used in conjunction with petal projections to produce interesting knot invariants?
\end{enumerate}

We hope that this paper encourages continued interest in the possibilities of petal projections. Viewing knots via their petal projections has the potential lead to more connections between knot theory, combinatorics, and algebra, just as it has led to connections between knot theory and probability theory in \cite{random}.

\section*{Acknowledgements} 

The authors would like to thank the Summer Institute of Mathematics at the University of Washington (SIMUW) for bringing us together to work on this project. In particular, we would like to thank S\'{a}ndor Kov\'{a}cs, Ron Irving, Jim Morrow, Kelly Emmrich, Shea Engle, Guangqiu Liang, Wilson Ly, and Andreas Quist for their leadership and support. We would also like to thank Andrew Lebedinsky and Gregory Baimetov, SIMUW participants, for contributing early ideas to the project. The authors would also like to thank the Simons Foundation (\#426566,
Allison Henrich) for their support of this research. Finally, the authors would like to thank their anonymous reviewer from \textit{Involve} for valuable feedback and detailed suggestions.

\appendix 
\section{Code}\label{app:code}
Any of the code we used in calculations for creating our appendices is included here. All of our code will also be available on \href{https://github.com/RobinTruax/PetalProjections}{Github (clickable)} at the URL \url{github.com/RobinTruax/PetalProjections}. 
\begin{lstlisting}[language=Python, caption=Contains an implementation of our algorithm]
#Importing the math library, SymPy (for prime factorization of natural numbers), and Numpy (for determinant calculation)
import math
import numpy
from sympy.ntheory import factorint

#Auxiliary Function 1: r_n(k)
def r(k,n): 
    return k%n

#Auxiliary Function 2: d(k), which depends on the petal number |P|
def d(k, petal_number): 
    if 2*r(k, petal_number-3) < (petal_number-3): 
        return petal_number - 3 - 2*r(k, petal_number-3)
    return petal_number - 5 - 2*r(k, petal_number-3)

#Auxilary Function 3: b(k), which depends on the petal number |P|
def b(k, petal_number): 
    return math.floor(k/(petal_number - 3))

#Auxiliary Function 4: b'(k), which depends on the petal number |P|
def bPrime(k, petal_number): 
    return r(b(k, petal_number)+d(k, petal_number),petal_number)

#Building Block Function 1: Generates the unsigned Gauss Code, depends only on the petal number |P|
def unsignedGaussCode(petal_number): 
    L_array = [i for i in range(int((petal_number-3)/2))] + [int((petal_number-5)/2) - i for i in range(int((petal_number-3)/2))]
    a_array = [0 for i in range(int((petal_number-3)/2))] + [i+1 for i in range(int((petal_number-3)/2))]
    n_array = [r(int((petal_number-1)/2)*math.floor(k/(petal_number-3)) + a_array[r(k,petal_number-3)],petal_number)+1 for k in range(petal_number*(petal_number-3))]

    unsigned_gauss_code = [petal_number*L_array[r(k, petal_number -3)] + n_array[k] for k in range(petal_number*(petal_number-3))]
    return unsigned_gauss_code

#Building Block Function 2: Generates the signed Gauss Code, requires unsigned Gauss code and petal permutation
def signGaussCode(Gauss_code, petal_permutation): 
    petal_number = len(petal_permutation)
    for k in range(petal_number*(petal_number - 3)): 
        if(petal_permutation[b(k, petal_number)] < petal_permutation[bPrime(k, petal_number)]): 
            pass
        else: 
            Gauss_code[k] *= -1
    return Gauss_code

#Building Block Function 3: Splits up a signed Gauss code into an aray of arcs
def splitSignedGaussCode(signed_Gauss_code): 
    #Preparing to split up the signed Gauss code into arcs
    arcs = []
    current_arc = []

    #Creating a list to scan through that starts with the first negative number
    trawl = list(range(len(signed_Gauss_code)))
    for i in range(len(signed_Gauss_code)): 
        if signed_Gauss_code[i] < 0: 
            trawl = trawl[i:] + trawl[:i]
            break

    #Splitting the signed Gauss code up into arcs
    current_arc.append(signed_Gauss_code[trawl[0]])
    for i in trawl[1:]: 
        current_arc.append(signed_Gauss_code[i])
        if signed_Gauss_code[i] < 0: 
            arcs.append(current_arc)
            current_arc = []
            current_arc.append(signed_Gauss_code[i])
    current_arc.append(signed_Gauss_code[trawl[0]])
    arcs.append(current_arc)

    return(arcs)

#Building Block Function 4: Generates the crossing matrix, requires the signed Gauss code and petal permutation
def createCrossingMatrix(signed_Gauss_code, petal_permutation): 
    #Initializing everything
    petal_number = len(petal_permutation)
    matrix_size = math.floor(petal_number*(petal_number-3)/2)
    crossing_matrix = [[0 for i in range(matrix_size)] for i in range(matrix_size)]

    arcs = splitSignedGaussCode(signed_Gauss_code)

    for crossing in range(1, len(crossing_matrix)+1): 
        for arc in range(len(arcs)): 
            if crossing in arcs[arc]: 
                crossing_matrix[crossing-1][arc] = 2
            if -1*crossing in arcs[arc]: 
                crossing_matrix[crossing-1][arc] = -1

    return crossing_matrix

#Building Block Function 5: Evaluates the first minor of the crossing matrix
def evaluateKnotDeterminant(crossing_matrix): 
    #This part slices off the first crossing (first step of getting the minor)
    crossing_matrix = crossing_matrix[1:]

    #This part slices off the first arc in each crossing (second step of getting the minor)
    for crossing in range(len(crossing_matrix)): 
        crossing_matrix[crossing] = crossing_matrix[crossing][1:]
    
    #This part converts our 2D array into a numpy array and then evaluates the determinant
    crossing_numpy_matrix = numpy.array(crossing_matrix)
    return abs(int(round(numpy.linalg.det(crossing_numpy_matrix))))

#This function directly generates the knot determinant of a petal projection with a given petal permutation
def knotDeterminant(petal_permutation): 
    unsigned = unsignedGaussCode(len(petal_permutation))
    signed = signGaussCode(unsigned, petal_permutation)
    crossing_matrix = createCrossingMatrix(signed, petal_permutation)
    return evaluateKnotDeterminant(crossing_matrix)

#This function presents what the knot determinant implies about the possible colorings of the knot
def presentDeterminant(determinant): 
    print("The knot's determinant is " + str(determinant) + ".")
    factorization = factorint(determinant)
    for i in factorization: 
        print("Since " + str(i) + " appears " + str(factorization[i]) + " time(s) in the prime factorization of " + str(determinant) + ", there are " + str(i) + "^(" + str(factorization[i]) + "+1) - " + str(i) + " = " + str(pow(i,factorization[i]+1)-i) + " nontrivial " + str(i) + "-colorings of the knot.")
    print("These are all nontrivial colorings of the knot, though there are p trivial colorings for every prime p.")
\end{lstlisting}
To test it, we use the following code: 
\begin{lstlisting}[language=Python, caption=Code for testing the algorithm by using the Stevedore knot.]
#Here, we import the code from Listing 1.
import spp_library

#The petal permutation of the Stevedore knot.
petal_permutation = [1, 3, 5, 2, 8, 4, 6, 9, 7]
petal_number = len(petal_permutation)
unsigned = spp_library.unsignedGaussCode(petal_number)
signed = spp_library.signGaussCode(unsigned, petal_permutation)
arcs = spp_library.splitSignedGaussCode(signed)
crossing_matrix = spp_library.createCrossingMatrix(signed, petal_permutation)
knotDeterminant = spp_library.evaluateKnotDeterminant(crossing_matrix)

print("")
print("You asked for analysis of the knot with petal permutation " + str(petal_permutation))
spp_library.presentDeterminant(knotDeterminant)
print("")
\end{lstlisting}

We also include the following code, which we used with minor modifications in the calculations for our other appendices. This code tests a list of petal projections and returns their knot determinant and the runtime (in ms) of each calculation.
\begin{lstlisting}[language=Python, caption=Code for testing the algorithm on sets of petal projections]
#Here, we import the code from Listing 1 and an algorithm timing library in Python.
import spp_library
import timeit
petal_projections_to_test = [[1, 3, 5, 2, 4],[1, 3, 5, 2, 7, 4, 6],[1, 3, 5, 2, 8, 4, 6, 9, 7]]
number_of_runs_each = 100

def mean(runtimes): 
    return sum(runtimes) / len(runtimes) 

for petal_projection in petal_projections_to_test: 
    #Compute the knot determinant
    petal_projection_determinant = spp_library.knotDeterminant(petal_projection)

    alg_runtimes = []
    #Run the petal projection calculation number_of_runs_each times and get a runtime list
    for i in range(number_of_runs_each): 
        start = timeit.default_timer()
        spp_library.knotDeterminant(petal_projection)
        alg_runtimes.append(1000*(timeit.default_timer() - start))

    #Print everything
    print("Petal Projection: " + str(petal_projection) + " | " + str(petal_projection_determinant) + "&" + str(round(min(alg_runtimes),4)) + "&" + str(round(mean(alg_runtimes),4)))
\end{lstlisting}

Note: all tests were conducted using the same computer, which uses a Ryzen 5 1600 CPU and GTX 1060 graphics card. Given an optimized and/or more powerful computer, a quicker language (ex. C/C++), or both, all of these runtimes could be significantly reduced. Additionally, this code is not optimized to make use of multithreading -- an implementation which did take advantage of multithreading could be orders of magnitude faster.

\section{The \textit{p}-colorability of All Prime Knots with Fewer Than 10 Crossings}\label{app:prime}
Again, we source our list of the prime knots with fewer than $10$ crossings in petal projection from \cite{adams1}. The minimum and average runtimes are collected from a total of $100$ attempts per knot in the list.
\newpage

\begin{center}
\begin{tabular}{|c|c|c|c|c|c|}
 \hline 
 Knot & $p(K)$ & A minimal petal representation & Determinant & Min. runtime (ms) & Avg. runtime (ms) \\ 
  \hline 
 $3_1$ & 5 & (1, 3, 5, 2, 4) & 3 & 0.0737 & 0.0769 \\
 $4_1$ & 7 & (1, 3, 5, 2, 7, 4, 6) & 5 & 0.1931 & 0.1990 \\
 $5_1$ & 7 & (1, 3, 6, 2, 5, 7, 4) & 5 & 0.1928 & 0.2009 \\
 $5_2$ & 7 & (1, 3, 6, 2, 4, 7, 5) & 7 & 0.1992 & 0.2040 \\
 $6_1$ & 9 & (1, 3, 5, 2, 8, 4, 6, 9, 7) & 9 & 0.5329 & 0.5821 \\
 $6_2$ & 9 & (1, 3, 5, 2, 8, 4, 7, 9, 6) & 11 & 0.5295 & 0.5968 \\
 $6_3$ & 9 & (1, 3, 5, 2, 9, 7, 4, 8, 6) & 13 & 0.5341 & 0.6007 \\
 $7_1$ & 9 & (1, 8, 4, 9, 5, 3, 6, 2, 7) & 7 & 0.5417 & 0.5917 \\
 $7_2$ & 9 & (1, 3, 6, 9, 7, 2, 4, 8, 5) & 11 & 0.5366 & 0.5845 \\
 $7_3$ & 9 & (1, 3, 6, 9, 7, 2, 5, 8, 4) & 13 & 0.5317 & 0.5887 \\
 $7_4$ & 9 & (1, 3, 6, 4, 8, 2, 5, 9, 7) & 15 & 0.5298 & 0.5768 \\
 $7_5$ & 9 & (1, 3, 6, 4, 8, 2, 7, 9, 5) & 17 & 0.5414 & 0.5953 \\
 $7_6$ & 9 & (1, 3, 6, 4, 9, 7, 2, 8, 5) & 19 & 0.5343 & 0.6008 \\
 $7_7$ & 9 & (1, 3, 7, 9, 4, 6, 2, 8, 5) & 21 & 0.5268 & 0.5884 \\
 $8_1$ & 11 & (1, 3, 5, 2, 8, 11, 9, 4, 6, 10, 7) & 13 & 1.5281 & 2.9710 \\
 $8_2$ & 11 & (1, 3, 5, 2, 9, 4, 7, 11, 8, 10, 6) & 17 & 1.5175 & 2.8166  \\
 $8_3$ & 11 & (1, 3, 6, 2, 9, 5, 11, 4, 7, 10, 8) & 17 & 1.5194 & 2.8753  \\
 $8_4$ & 11 & (1, 3, 5, 8, 6, 2, 10, 4, 7, 11, 9) & 19 & 1.4199 & 3.0731  \\
 $8_5$ & 11 & (1, 3, 5, 8, 6, 11, 9, 2, 10, 4, 7) & 21 & 1.5375 & 3.0680  \\
 $8_6$ & 11 & (1, 3, 5, 2, 8, 6, 10, 4, 9, 11, 7) & 23 & 1.5763 & 3.8458  \\
 $8_7$ & 11 & (1, 3, 5, 2, 10, 7, 4, 8, 11, 9, 6) & 23 & 1.5653 & 3.5296  \\
 $8_8$ & 11 & (1, 3, 5, 2, 8, 6, 11, 9, 4, 10, 7) & 25 & 1.4279 & 3.3398  \\
 $8_9$ & 11 & (1, 3, 5, 9, 2, 7, 11, 6, 4, 10, 8) & 25 & 1.6110 & 3.6660  \\
 $8_{10}$ & 11 & (1, 3, 5, 2, 9, 7, 11, 8, 4, 10, 6) & 27 & 1.5062 & 3.5207  \\
 $8_{11}$ & 11 & (1, 3, 5, 2, 8, 11, 9, 4, 7, 10, 6) & 27 & 1.6855 & 3.5786  \\
 $8_{12}$ & 11 & (1, 3, 5, 2, 9, 11, 8, 6, 10, 4, 7) & 29 & 1.4598 & 3.4152  \\
 $8_{13}$ & 11 & (1, 3, 5, 2, 10, 7, 4, 9, 11, 8, 6) & 29 & 1.5943 & 3.3468  \\
 $8_{14}$ & 11 & (1, 3, 5, 2, 10, 8, 11, 6, 9, 4, 7) & 31 & 1.8390 & 3.8781  \\
 $8_{15}$ & 11 & (1, 3, 5, 2, 8, 11, 7, 9, 4, 10, 6) & 33 & 1.5407 & 3.5597  \\
 $8_{16}$ & 11 & (1, 3, 5, 8, 6, 2, 11, 9, 4, 10, 7) & 35 & 1.6293 & 3.7273  \\
 $8_{17}$ & 11 & (1, 3, 5, 8, 6, 2, 10, 4, 9, 11, 7) & 37 & 1.4885 & 3.4082  \\
 $8_{18}$ & 11 & (1, 3, 7, 4, 10, 2, 8, 6, 11, 9, 5) & 45 & 1.6567 & 3.6826  \\
 \hline
\end{tabular}
\end{center}

\begin{center}
 \begin{tabular}{|c|c|c|c|c|c|}
  \hline 
 Knot & $p(K)$ & A minimal petal representation & Determinant & Min. runtime (ms) & Avg. runtime (ms) \\ 
 \hline 
 $8_{19}$ & 7 & (1, 4, 7, 3, 6, 2, 5) & 3&0.1811&0.1994\\
 $8_{20}$ & 9 & (1, 3, 5, 8, 2, 6, 9, 4, 7) & 9&0.5331&0.5799\\
 $8_{21}$ & 9 & (1, 3, 5, 8, 2, 7, 4, 9, 6) & 15&0.5282&0.5863\\
 $9_{1}$ & 11 & (1, 10, 5, 11, 6, 4, 7, 3, 8, 2, 9) & 9&1.5855&3.3514\\
 $9_{2}$ & 11 & (1, 3, 6, 10, 7, 2, 4, 8, 11, 9, 5) & 15&1.4888&3.0756\\
 $9_{3}$ & 11 & (1, 3, 7, 5, 9, 2, 6, 11, 8, 10, 4) & 19&1.4289&3.0935\\
 $9_{4}$ & 11 & (1, 3, 6, 10, 7, 2, 5, 8, 11, 9, 4) & 21&1.5295&3.4236\\
 $9_{5}$ & 11 & (1, 3, 6, 4, 8, 11, 9, 2, 5, 10, 7) & 23&1.4775&2.4631\\
 $9_{6}$ & 11 & (1, 3, 6, 4, 9, 2, 7, 11, 8, 10, 5) & 27&1.4985&2.6101\\
 $9_{7}$ & 11 & (1, 3, 6, 10, 7, 2, 4, 9, 11, 8, 5) & 29&1.6424&3.4747\\
 $9_{8}$ & 11 & (1, 3, 6, 10, 8, 4, 11, 7, 2, 9, 5) & 31&1.5704&3.2222\\
 $9_{9}$ & 11 & (1, 3, 6, 10, 7, 2, 5, 9, 11, 8, 4) & 31&1.4988&3.2292\\
 $9_{10}$ & 11 & (1, 3, 7, 5, 8, 11, 9, 2, 6, 10, 4) & 33&1.4525&3.3888\\
 $9_{11}$ & 11 & (1, 3, 6, 4, 10, 7, 2, 8, 11, 9, 5) & 33&1.6525&3.2743\\
 $9_{12}$ & 11 & (1, 3, 6, 10, 5, 7, 2, 8, 11, 9, 4) & 35&1.6505&3.5714\\
 $9_{13}$ & 11 & (1, 3, 6, 4, 9, 2, 5, 11, 8, 10, 7) & 37&1.5017&3.437\\
 $9_{14}$ & 11 & (1, 3, 7, 10, 5, 2, 9, 11, 8, 4, 6) & 37&1.6414&3.6194\\
 $9_{15}$ & 11 & (1, 3, 6, 4, 10, 8, 2, 7, 11, 9, 5) & 39&1.5063&3.4034\\
 $9_{16}$ & 11 & (1, 3, 7, 4, 10, 2, 9, 11, 6, 8, 5) & 39&1.6249&3.6076\\
 $9_{17}$ & 11 & (1, 3, 7, 10, 4, 6, 2, 9, 11, 8, 5) & 39&1.5607&3.9702\\
 $9_{18}$ & 11 & (1, 3, 6, 4, 8, 11, 9, 2, 7, 10, 5) & 41&1.4809&3.4053\\
 $9_{19}$ & 11 & (1, 3, 7, 5, 9, 11, 4, 8, 2, 10, 6) & 41&1.5197&3.4281\\
 $9_{20}$ & 11 & (1, 3, 6, 4, 10, 8, 2, 9, 5, 11, 7) & 41&1.4670&3.3930\\
 $9_{21}$ & 11 & (1, 3, 6, 4, 10, 7, 2, 9, 11, 8, 5) & 43&1.4145&3.1027\\
 $9_{22}$ & 11 & (1, 3, 6, 4, 9, 7, 2, 10, 5, 11, 8) & 43&1.7543&3.6143\\
 $9_{23}$ & 11 & (1, 3, 6, 4, 9, 11, 7, 2, 8, 10, 5) & 45&1.5742&3.742\\
 $9_{24}$ & 11 & (1, 3, 6, 11, 5, 7, 2, 9, 4, 10, 8) & 45&1.5370&3.3379\\
 $9_{25}$ & 11 & (1, 3, 6, 4, 8, 11, 7, 9, 2, 10, 5) & 47&1.5600&3.6257\\
 $9_{26}$ & 11 & (1, 3, 7, 5, 10, 6, 2, 9, 11, 4, 8) & 47&1.6188&3.7474\\
 $9_{27}$ & 11 & (1, 3, 6, 4, 11, 7, 2, 8, 10, 5, 9) & 49&1.7437&3.6646\\
 $9_{28}$ & 11 & (1, 3, 6, 11, 5, 7, 2, 8, 10, 4, 9) & 51&1.6539&3.8004\\
 $9_{29}$ & 11 & (1, 3, 6, 4, 10, 7, 2, 8, 5, 11, 9) & 51&1.5392&3.9865\\
 $9_{30}$ & 11 & (1, 3, 6, 4, 10, 8, 2, 7, 11, 5, 9) & 53&1.7512&4.2598\\
 $9_{31}$ & 11 & (1, 3, 6, 10, 5, 7, 2, 9, 11, 4, 8) & 55&1.4685&3.7721\\
 $9_{32}$ & 11 & (1, 3, 6, 4, 9, 11, 7, 2, 10, 5, 8) & 59&1.6178&3.5188\\
 $9_{33}$ & 11 & (1, 3, 6, 4, 10, 2, 7, 11, 9, 5, 8) & 61&1.4888&3.6090\\
 $9_{34}$ & 13 & (1, 3, 7, 9, 13, 5, 11, 8, 2, 4, 6, 10, 12) & 69&4.5353&11.4473\\
 $9_{35}$ & 11 & (1, 3, 10, 6, 2, 9, 11, 8, 5, 7, 4) & 27&1.4939&2.7999\\
 $9_{36}$ & 11 & (1, 3, 6, 4, 9, 7, 11, 8, 2, 10, 5) & 37&1.5997&4.0018\\
 $9_{37}$ & 11 & (1, 3, 7, 10, 4, 6, 2, 8, 11, 9, 5) & 45&1.7504&3.8878\\
 $9_{38}$ & 11 & (1, 3, 6, 4, 9, 2, 7, 11, 5, 10, 8) & 57&1.6299&3.6175\\
 $9_{39}$ & 11 & (1, 3, 6, 4, 10, 2, 7, 9, 5, 11, 8) & 55&1.5260&3.7586\\
 $9_{40}$ & 13 & (1, 11, 7, 5, 13, 2, 10, 8, 6, 12, 4, 9, 3) & 75&5.2471&12.1151\\
 $9_{41}$ & 11 & (1, 3, 7, 11, 4, 8, 10, 6, 2, 9, 5) & 49&1.5291&3.5573\\
 $9_{42}$ & 9 & (1, 3, 6, 2, 9, 5, 8, 4, 7) & 7&0.5351&0.5776\\
 $9_{43}$ & 9 & (1, 3, 6, 9, 5, 8, 2, 7, 4) & 13&0.5262&0.6009\\
 $9_{44}$ & 9 & (1, 3, 6, 9, 4, 7, 2, 8, 5) & 17&0.5248&0.5718\\
 $9_{45}$ & 9 & (1, 3, 7, 4, 9, 6, 2, 8, 5) & 23&0.5282&0.5999\\
 $9_{46}$ & 9 & (1, 3, 6, 9, 5, 2, 8, 4, 7) & 9&0.5362&0.6032\\
 $9_{47}$ & 11 & (1, 3, 5, 7, 10, 4, 9, 6, 2, 11, 8) & 27&1.5552&3.4711\\
 $9_{48}$ & 11 & (1, 3, 5, 2, 9, 11, 7, 4, 10, 6, 8) & 27&1.6477&3.7317\\
 $9_{49}$ & 11 & (1, 3, 5, 2, 7, 11, 8, 4, 10, 6, 9) & 25&1.5622&3.8258\\
  \hline
\end{tabular}
\end{center}
\section{Finding the Limits of the Algorithm}\label{app:limits}
First, we test our program 1000 times against the petal projection $(1,2,\dots,n)$ for various (odd) $n$. Note that the result will always be $1$, as this petal projection degenerates into the unknot via petal cancellation.\\

\begin{center}
\begin{tabular}{|c|c|c|}
 \hline 
 $n$ & Min. runtime (ms) & Avg. runtime (ms) \\ 
  \hline 
 $11$ & 1.3715 & 2.0532  \\
 $25$ & 34.1978 & 60.7494  \\
 $51$ & 501.0886 & 549.8304  \\
 $75$ & 2255.2766 & 2341.1118  \\
 $101$ & 7693.6962 & 7919.2274  \\
 \hline
\end{tabular}
\end{center}

Next, we test our program against 1000 random petal permutations with petal number $n$ -- that is, random elements of $S_n$ for various (odd) $n$. Here we expect a variety of determinants, but we cannot list all of them here.\\

\begin{center}
\begin{tabular}{|c|c|c|}
 \hline 
 $n$ & Min. runtime (ms) & Avg. runtime (ms) \\ 
  \hline 
 $11$ & 1.5544 & 3.7661  \\
 $25$ & 45.0358 & 89.2748  \\
 $51$ & 513.6894 & 576.6921  \\
 $75$ & 2320.5901 & 2425.1528  \\
 $101$ & 8086.1688 & 8339.2420  \\
 \hline
\end{tabular}
\end{center}

Clearly, this algorithm suffices to compute the determinant of any ``reasonably-sized" petal projection in a fairly short amount of time.

\section{Distribution of \textit{p}-colorability Amongst Randomly Generated Petal Projections}\label{app:dist}

Finally, we will offer an empirical description of the proportion of petal permutations with petal number $n$ that are $p$-colorable for $p \leq 23$. For example, the $n=5$, $p=3$ entry is 8.3\%, so approximately 8.3\% of the $5$-petal projections are $3$-colorable. The ``NC" column refers to the approximate percentage of petal permutations with petal number $n$ that are not $p$-colorable for any $p$.\\

For sufficiently small $n$ (i.e. $n$ such that $n! < 1000$), we simply test every permutation of $\{1, \dots, n\}$. Otherwise, we select $1000$ random permutations of $\{1, \dots, n\}$, find the knot determinant of the corresponding petal projection, and calculate what percentage are $p$-colorable for each $p$ that appears in the prime factorization of any of them. This is a \textit{very} rough approximation: we simply use this for intuition and to generate further questions related to petal projection determinants.
\newpage

\begin{center}
\begin{tabular}{|c|c|c|c|c|c|c|c|c|c|c|c|c|c|c|c|}
 \hline 
 $\mathbf{n}$ & NC & $\bf 3$ & $\bf 5$ & $\bf 7$ & $\bf 11$ & $\bf 13$ & $\bf 17$ & $\bf 19$ & $\bf 23$ \\ 
 \hline 
{\bf 1} & 100 & 0.0 & 0.0 & 0.0 & 0.0 & 0.0 & 0.0 & 0.0 & 0.0 \\
{\bf 3} & 100 & 0.0 & 0.0 & 0.0 & 0.0 & 0.0 & 0.0 & 0.0 & 0.0 \\
{\bf 5} & 91.7 & 8.3 & 0.0 & 0.0 & 0.0 & 0.0 & 0.0 & 0.0 & 0.0 \\
{\bf 7} & 64.5 & 22.4 & 6.3 & 5.0 & 0.2 & 0.2 & 0.0 & 0.0 & 0.0 \\
{\bf 9} & 50.2 & 25.0 & 10.2 & 7.3 & 2.7 & 1.1 & 0.6 & 0.2 & 0.2 \\
{\bf 11} & 38.8 & 30.0 & 16.0 & 7.6 & 4.6 & 2.6 & 1.7 & 1.1 & 0.8 \\
{\bf 13} & 24.6 & 32.3 & 17.6 & 11.6 & 5.2 & 4.1 & 2.8 & 2.1 & 1.4 \\
{\bf 15} & 20.3 & 37.0 & 20.3 & 12.7 & 6.4 & 4.7 & 4.3 & 3.0 & 2.4 \\
{\bf 17} & 11.8 & 33.1 & 17.7 & 14.1 & 8.6 & 6.5 & 4.9 & 4.6 & 3.6 \\
{\bf 19} & 8.7 & 34.5 & 20.1 & 14.6 & 8.1 & 8.4 & 4.2 & 4.6 & 2.5 \\
{\bf 21} & 5.5 & 36.2 & 20.3 & 16.3 & 8.0 & 6.5 & 5.8 & 4.1 & 3.1 \\
{\bf 23} & 2.7 & 37.5 & 21.2 & 15.2 & 9.4 & 8.2 & 5.1 & 4.9 & 4.6 \\
{\bf 25} & 1.8 & 35.7 & 22.3 & 12.5 & 8.4 & 6.4 & 4.3 & 4.1 & 3.7 \\
{\bf 27} & 0.5 & 35.6 & 22.2 & 15.8 & 9.2 & 8.2 & 5.3 & 5.1 & 5.7 \\
{\bf 29} & 0.5 & 37.2 & 21.6 & 16.9 & 8.8 & 8.3 & 6.5 & 5.8 & 4.1 \\
{\bf 31 } & 0.2 & 35.4 & 19.1 & 17.4 & 8.4 & 8.3 & 6.0 & 5.1 & 4.0 \\
{\bf 33 } & $\sim 0.0$ & 35.4 & 21.1 & 12.0 & 9.6 & 9.0 & 6.0 & 5.4 & 4.6 \\
{\bf 35 } & $\sim 0.1$ & 37.0 & 21.5 & 15.2 & 9.5 & 8.3 & 5.3 & 4.8 & 4.0 \\
{\bf 37 } & $\sim 0.0$ & 36.8 & 22.2 & 15.3 & 10.8 & 5.9 & 5.5 & 3.9 & 4.8 \\
{\bf 39 } & $\sim 0.0$ & 35.9 & 20.6 & 16.7 & 9.5 & 7.2 & 6.5 & 3.8 & 4.6 \\
{\bf 41 } & $\sim 0.0$ & 34.8 & 19.5 & 16.0 & 10.5 & 6.3 & 5.4 & 6.1 & 4.9 \\
{\bf 43 } & $\sim 0.0$ & 38.3 & 21.5 & 13.8 & 8.9 & 9.3 & 5.6 & 5.1 & 4.4 \\
{\bf 45 } & $\sim 0.0$ & 35.7 & 20.1 & 13.5 & 8.4 & 6.7 & 6.1 & 5.0 & 4.2 \\
{\bf 47 } & $\sim 0.0$ & 35.1 & 20.7 & 13.5 & 8.5 & 8.3 & 5.3 & 5.2 & 3.7 \\
{\bf 49 } & $\sim 0.0$ & 36.7 & 22.2 & 14.2 & 10.1 & 8.2 & 5.5 & 5.6 & 4.2 \\
 \hline
\end{tabular}
\end{center}
\end{document}